\documentclass[11pt,reqno]{amsart}
\usepackage{amssymb,mathrsfs,graphicx,subfigure, enumerate}
\usepackage{amsmath,amsfonts,amssymb,amscd,amsthm,bbm}
\usepackage{extpfeil}
\usepackage{graphicx,colortbl}
\usepackage{float}
\usepackage{epsfig}
\usepackage{caption}
\usepackage{bm}
\graphicspath{{Figure/}}
\usepackage{kotex}
\usepackage[margin=1in]{geometry}

\makeatletter
\@namedef{subjclassname@2020}{\textup{2020} Mathematics Subject Classification}
\makeatother

\title[Spherically symmetric Navier-Stokes-Korteweg equations]{Stationary solutions to the spherically symmetric compressible fluid with capillarity effect}

\author[Kim]{Jeongho Kim}
\address[Jeongho Kim]{\newline Department of Applied Mathematics, \newline Kyung Hee University, 1732, Deogyeong-daero, Giheung-gu, Yongin-si, Gyeonggi-do 17104, Republic of Korea}
\email{jeonghokim@khu.ac.kr}

\begin{document}
\newtheorem{theorem}{Theorem}[section]
\newtheorem{lemma}{Lemma}[section]
\newtheorem{corollary}{Corollary}[section]
\newtheorem{proposition}{Proposition}[section]
\newtheorem{remark}{Remark}[section]
\newtheorem{definition}{Definition}[section]
\renewcommand{\theequation}{\thesection.\arabic{equation}}
\renewcommand{\thetheorem}{\thesection.\arabic{theorem}}
\renewcommand{\thelemma}{\thesection.\arabic{lemma}}
\newcommand{\bbr}{\mathbb R}
\newcommand{\R}{\mathbb{R}}
\newcommand{\e}{\varepsilon}
\newcommand{\pa}{\partial}
\newcommand{\tU}{\widetilde{U}}
\newcommand{\tu}{\widetilde{u}}
\newcommand{\tv}{\widetilde{v}}
\newcommand{\tw}{\widetilde{w}}
\newcommand{\pv}{p(v)}
\newcommand{\tp}{\widetilde{p}}
\newcommand{\tpv}{p(\widetilde{v})}
\newcommand{\norm}[1]{\left\lVert#1\right\rVert}
\newcommand{\tuSX}{(\widetilde{u}^S)^X}
\newcommand{\tvSX}{(\widetilde{v}^S)^X}
\newcommand{\twSX}{(\widetilde{w}^S)^X}
\newcommand{\tUSX}{(\widetilde{U}^S)^X}
\newcommand{\tm}{\widetilde{\mu(v)}}
\newcommand{\tk}{\widetilde{\kappa(v)}}



\thanks{}

\begin{abstract}
	We consider the spherically symmetric Navier--Stokes--Korteweg (NSK) system on the exterior domain $\Omega=\{x\in\R^n~|~|x|>1\}$ with $n\ge2$ when the boundary and far-field data are given. We show that, if the boundary data are sufficiently small, then there exists a unique smooth stationary solution to the spherically symmetric NSK system with impermeable wall, inflow, and outflow boundary conditions. We also establish the decay rate of the stationary solutions. Precisely, the stationary solution for the impermeable wall problem exponentially decays to the far-field states, while that of the inflow/outflow problem algebraically decays. Finally, we investigate the asymptotic convergences of the stationary solution for the impermeable wall problem as the capillarity coefficient vanishes. Numerical results validate that our theoretical convergence rate of the stationary solution is optimal.
\end{abstract}

\maketitle

\section{Introduction}\label{sec:1}
\setcounter{equation}{0}

In this paper, we consider a viscous compressible fluid with capillary effect on the domain $\Omega\subset\R^n$ with $n>1$, whose dynamics is governed by the so-called compressible Navier--Stokes--Korteweg (NSK) equations:
\begin{align}
\begin{aligned}\label{NSK}
&\rho_t +\nabla\cdot(\rho U) = 0,\quad t>0,\quad x\in\Omega,\\
&\rho(U_t +(U\cdot\nabla U))+\nabla P(\rho) = \nu\Delta U + (\nu+\lambda)\nabla(\nabla\cdot U)+\kappa\rho\nabla\Delta\rho,
\end{aligned}
\end{align}
where $\rho$ and $U$ are the density and velocity of the fluid, $\nu$ and $\lambda$ are viscosity coefficients satisfying $\nu>0$ and $2\nu+n\lambda\ge0$, and $\kappa>0$ is a capillarity coefficient. The pressure $P(\rho)$ is assumed to be polytropic form $P(\rho) = \rho^\gamma$ with the adiabatic constant $\gamma\ge1$. 

The equation of motion for the fluid with capillarity was devised by the early work of van der Waals \cite{W94} and Korteweg \cite{K01}, and the modern form of the equation \eqref{NSK} was rigorously derived by Dunn and Serrin \cite{DS85} by considering thermodynamics of interstitial work. After the NSK model was introduced, various topics such as well-posedness of local and global solutions \cite{CH13,DD01,GL16,H11,H17,HL96,J10,LV18}, stability of nonlinear waves \cite{C12,CHZ15,FLP26,FPZ23,HK26,HKKL25}, and vanishing viscosity and/or capillarity limit \cite{CH13,EKKpre,GL16} has been investigated in numerous literature. 

In this paper, we are interested in the spherically symmetric solution to \eqref{NSK} when the domain is given by an exterior domain $\Omega:=\{x\in\R^n~|~|x|>1\}$. Substituting the spherically symmetric functions $\rho$ and $U$ of the form
\[\rho(t,x) = \rho(t,r),\quad U(t,x) = \frac{x}{r}u(t,r),\quad r:=|x|\]
into \eqref{NSK}, we derive the equations for $(\rho,u)$ as
\begin{align}
\begin{aligned}\label{NSK_radial}
&\rho_t + \frac{(r^{n-1}\rho u)_r}{r^{n-1}}=0,\quad t>0,\quad r\in(1,\infty)\\
&\rho(u_t+uu_r)+P(\rho)_r = \mu\left(\frac{(r^{n-1}u)_r}{r^{n-1}}\right)_r+\kappa\rho \left(\rho_{rr}+\frac{n-1}{r}\rho_r\right)_r.
\end{aligned}
\end{align}
Here, $\mu :=2\nu+\lambda\ge0$ is a positive constant. We consider an initial-boundary value problem \eqref{NSK_radial} subject to the initial data $(\rho(0,r),u(0,r))=(\rho_0(r),u_0(r))$ with $\rho_0(r)>0$ and the far-field condition
\[\lim_{r\to\infty}(\rho(t,r),u(t,r))=(\rho_+,u_+),\quad t>0,\]
where $\rho_+>0$ and $u_+\in\R$ are constants. Finally, the boundary condition at $r=1$ is suitably given, depending on the sign of the velocity at the boundary:
\[\begin{cases}
	u(t,1)=u_-,\quad \rho_r(t,1)=\rho_b,\quad& t>0,\quad\mbox{if}\quad u_-\le0,\\
	(\rho,u)(t,1)=(\rho_-,u_-),\quad \rho_r(t,1)=\rho_b,\quad& t>0,\quad\mbox{if}\quad u_->0,
\end{cases}\]
where $\rho_->0$, $u_-\in\R$ and $\rho_b\in\R$ are constants. 
When $u_->0$, the initial-boundary value problem is called the {\it inflow problem}. When $u_-=0$, we call it as the {\it impermeable wall problem}, and if $u_-<0$, then it is called the {\it outflow problem}. Note that, for the inflow problem, since the characteristic curve at the boundary $r=1$ is inward to the domain, we need to prescribe the boundary value $\rho(t,1)=\rho_-$. For the other two cases, we cannot impose the boundary value of $\rho$, and it is determined by the solution.

When the capillarity effect vanishes, that is when $\kappa=0$ in \eqref{NSK_radial}, it reduces to the standard spherically symmetric Navier--Stokes (NS) equations, whose existence of the stationary solution and the time-asymptotic behavior toward it have been extensively studied. For the impermeable wall problem, Jiang \cite{J96} proved that there exists a unique global-in-time classical solution to the spherically symmetric NS equations, and also showed the time asymptotic stability for the case of $n=3$. Later, Nakamura, Nishibata, and Yanagi \cite{NNY04} extended this result by proving the time-asymptotic stability for the large initial data and the presence of external potential force. In the case of heat-conducting fluids, Nakamura and Nishibata \cite{NN08} proved the large-time behavior under the presence of the potential force.

On the other hand, when the velocity at the boundary is non-zero, the existence and stationary solution and the time-asymptotic behavior toward it requires more delicate analysis. In \cite{HM21}, Hashimoto and Matsumura proved an existence of stationary solutions for both inflow and outflow problems for spherically symmetric NS equations. The time-asymptotic stability for the outflow problem was investigated in \cite{HNS24} under the small initial perturbation, and the case of large initial data was resolved in \cite{HNpre}. On the other hand, for the inflow problem, the time-asymptotic stability under the small initial perturbation was proved in \cite{HHNpre}. We also refer the readers to \cite{HMpre} for the existence of stationary solution to the inflow/outflow problems for the spherically symmetric Navier--Stokes--Fourier system and \cite{HM24} for the inviscid limit of the stationary solution to the NS equations.

However, to the extent of the author's knowledge, the spherically symmetric solution to the NSK equations has not been considered in the literature. Thus, the goal of the paper is to take a first step to extend these results on the NS equations to the compressible fluid with capillary effect. Precisely, our first aim is to show the existence of stationary solutions $(\tilde{\rho},\tilde{u})(r)$ to \eqref{NSK_radial}, which satisfies
\begin{align}
	\begin{aligned}\label{NSK_stationary}
	&\frac{1}{r^{n-1}}(r^{n-1}\tilde{\rho}\tilde{u})_r = 0,\quad r>1,\\
	&\tilde{\rho}\tilde{u}\tilde{u}_r +P(\tilde{\rho})_r=\mu\left(\frac{(r^{n-1}\tilde{u})_r}{r^{n-1}}\right)_r+\kappa\tilde{\rho}\left(\tilde{\rho}_{rr}+\frac{n-1}{r}\tilde{\rho}_r\right)_r,
	\end{aligned}
\end{align} 
subject to the far-field and boundary conditions:
\begin{equation}\label{boundary}
	\lim_{r\to\infty}(\tilde{\rho}(r),\tilde{u}(r))=(\rho_+,u_+),\quad\tilde{u}(1)=u_-,\quad\tilde{\rho}_r(1)=\rho_b.
\end{equation}
It follows from \eqref{NSK_stationary}$_1$ that
\begin{equation}\label{tildeu}
	r^{n-1}\tilde{\rho}(r)\tilde{u}(r) = \tilde{\rho}(1)u_-,\quad\mbox{that is}\quad \tilde{u}(r) = \frac{\tilde{\rho}(1)}{\tilde{\rho}(r)} u_- r^{1-n},\quad r\ge 1.
\end{equation}
Thus, the only possible value for the far-field velocity is $u_+=0$ for $n>1$. Furthermore, for the impermeable wall problem, we have $u_-=0$, which implies $\tilde{u}(r)=0$ for all $r\ge1$. Substituting $\tilde{u}\equiv0$ into \eqref{NSK_stationary}$_2$, we obtain the following nonlinear equation for $\tilde{\rho}$:
\begin{equation}\label{NSK_stationary_impermeable}
	P(\tilde{\rho})_r = \kappa\tilde{\rho}\left(\tilde{\rho}_{rr}+\frac{n-1}{r}\tilde{\rho}_r\right)_r,
\end{equation}
subject to the far-field and boundary conditions
\begin{equation}\label{boundary_impermeable}
	\tilde{\rho}_r(1) = \rho_b,\quad \lim_{r\to\infty}\tilde{\rho}(r) = \rho_+.
\end{equation}
The first result of the paper is an existence of the stationary solution for the impermeable wall problem.

\begin{theorem}[Impermeable wall problem]\label{thm:impermeable}
	Suppose $u_-=0$. Let $\sigma$ be any positive constant satisfying $0<\sigma<\sqrt{\frac{h'(\rho_+)}{\kappa}}$. Then, there exists a positive constant $\delta$ and $C=C(n,\kappa,\rho_+)$ such that the following holds. If $|\rho_b|<\delta$, there exists a unique smooth solution $\tilde{\rho}(r)$ to \eqref{NSK_stationary_impermeable}--\eqref{boundary_impermeable} satisfying
	\[|\tilde{\rho}(r)-\rho_+|\le C|\rho_b|e^{-\sigma r},\quad r\ge 1.\]
\end{theorem}

\begin{remark}
	Unlike the standard NS equations case, the stationary solution for the impermeable wall problem to the NSK system is non-trivial. 
\end{remark}

Now, we consider the inflow or outflow problem with $u_-\neq0$. In this case, the velocity is not identically zero, but it is determined by $\tilde{\rho}$ as in \eqref{tildeu}. The second result of the paper is an existence of stationary solution $(\tilde{\rho},\tilde{u})$ to \eqref{NSK_stationary} for both inflow and outflow cases.

\begin{theorem}[Inflow and outflow problems]\label{thm:inflow_outflow}
Suppose $u_-\neq 0$. Then, there exists a positive constant $\delta$ and $C=C(n,\kappa,\rho_+)$ such that the following holds. If $|\rho_b|,|u_-|<\delta$, then there exists a unique smooth solution $(\tilde{\rho},\tilde{u})$ to \eqref{NSK_stationary}--\eqref{boundary} satisfying
	\[|\tilde{\rho}(r)-\rho_+|\le C(|\rho_b|+|u_-|^2)r^{-2(n-1)},\quad |\tilde{\rho}_r(r)|\le C(|\rho_b|+|u_-|^2)r^{-(2n-1)},\quad r\ge 1\]
	and
	\[C^{-1}|u_-|r^{-(n-1)}\le |\tilde{u}(r)|\le C|u_-|r^{-(n-1)},\quad r\ge1.\]
\end{theorem}

\begin{remark}
	Note that we do not prescribe the value of $\tilde{\rho}(1)=\rho_-$ in the boundary conditions \eqref{boundary}, even for the inflow problem. The reason is that, even for the case of $u_->0$, we can uniquely determine the solution $(\tilde{\rho},\tilde{u})$ to \eqref{NSK_stationary} with the information $\tilde{\rho}_r(1)=\rho_b$ and $\lim_{r\to\infty}\tilde{\rho}(r)=\rho_+$. We refer to Section \ref{sec:inflow_outflow} for details. Therefore, for an arbitrary $\rho_-$, there might be no stationary solution $(\tilde{\rho}(r),\tilde{u}(r))$ satisfying both $\tilde{\rho}(1)=\rho_-$ and $\tilde{\rho}_r(1)=\rho_b$, and $\lim_{r\to\infty}\tilde{\rho}(r)=\rho_+$, and $\rho_-$ should be determined by $\rho_b$, $u_-$, $\rho_+$, and the solution itself. 
\end{remark}

\begin{remark}
	For the impermeable wall problem, the term with viscosity $\mu$ vanishes since $\tilde{u}$ is identically zero. Moreover, in the proof of Theorem \ref{thm:inflow_outflow}, the estimates are uniform in $\mu$ if it is bounded. In particular, Theorem \ref{thm:inflow_outflow} is valid even for $\mu=0$. Thus, our results covers the case of the spherically symmetric Euler--Korteweg system, the system \eqref{NSK_radial} with $\mu=0$.
\end{remark}

The final result of the paper is the asymptotic convergence of the stationary solution $\tilde{\rho}$ to the impermeable wall problem \eqref{NSK_stationary_impermeable}--\eqref{boundary_impermeable} as $\kappa$ decays to 0, called the vanishing capillarity limit. Precisely, we consider two types of asymptotic limit. We first consider the case when the boundary condition $\rho_b$ is fixed as a constant as $\kappa\to0$. In this case, one can easily expect that the asymptotic profile of the solution becomes a constant $\bar{\rho}(r)\equiv \rho_+$. On the other hand, if the boundary data is singularly scaled as $\rho_b = \rho_b^0/\sqrt{\kappa}$ for some fixed $\rho_b^0$, then it is expected that the limit profile of $\tilde{\rho}$ is related to the solution $\bar{\rho}$ to the following second-order boundary value problem:
\begin{equation}\label{eq:limit_0}
	\bar{\rho}_{yy}=h(\bar{\rho})-h(\rho_+),\quad \bar{\rho}_y(0)=\rho^0_b,\quad \lim_{y\to\infty}\bar{\rho}(y)=\rho_+,
\end{equation}
where $h(\rho)$ is defined in \eqref{eq:h}. We refer to Section \ref{sec:vanishing} for the detailed argument of these anticipations. Then, we have the following asymptotic convergence of the stationary solutions to the impermeable wall problem.

\begin{theorem}[Vanishing capillarity limit]\label{thm:vanishing_capillarity}
	Let $\tilde{\rho}^\kappa(r)$ be a unique stationary solution to the impermeable wall problem satisfying \eqref{NSK_stationary_impermeable}--\eqref{boundary_impermeable}.\\
	
	\noindent (1) (Fixed capillarity boundary condition) If $\rho_b$ is fixed constant for all $\kappa>0$, then, there exists a positive constant $C$ that is independent of $\kappa$ such that
	\[\|\tilde{\rho}^\kappa-\rho_+\|_{L^2_r(1,\infty)}\le C\kappa^{\frac{3}{4}},\quad \|\tilde{\rho}_r^\kappa\|_{L^2_r(1,\infty)}\le C\kappa^{\frac{1}{4}},\quad \|\tilde{\rho}^\kappa-\rho_+\|_{L^\infty(1,\infty)}\le C\kappa^{\frac{1}{2}}.\]
	\noindent (2) (Singularly scaled capillarity boundary condition) If $\rho_b=\rho_b^0/\sqrt{\kappa}$ for some sufficiently small fixed $\rho_b^0$, then there exists a unique solution $\bar{\rho}$ to \eqref{eq:limit_0}. Moreover, let $\bar{\rho}^\kappa(r):=\bar{\rho}(\frac{r-1}{\sqrt{\kappa}})$. Then, there exists a positive constant $C$ that is independent of $\kappa$ such that
	\begin{align*}
		\left\|\tilde{\rho}^\kappa-\bar{\rho}^\kappa\right\|_{L^2_r(1,\infty)}\le C\kappa^{\frac{3}{4}},\quad		\left\|(\tilde{\rho}^\kappa-\bar{\rho}^\kappa)_r\right\|_{L^2_r(1,\infty)}\le C\kappa^{\frac{1}{4}},\quad \left\|\tilde{\rho}^\kappa-\bar{\rho}^\kappa\right\|_{L^\infty(1,\infty)}\le C\kappa^{\frac{1}{2}}.
	\end{align*}
	Here, the $L^2_r$-norm of the function defined on $[1,\infty)$ is given as
	\[\|f\|_{L^2_r(1,\infty)}:=\left(\int_1^\infty |f(r)|^2r^{n-1}\,dr\right)^{1/2}.\]
\end{theorem}

\begin{remark}\label{rem:y}
	The result of Theorem \ref{thm:vanishing_capillarity} (2) can be represented in terms of the rescaled variable $y$ as follows. Since $r = 1+\sqrt{\kappa}y$, we apply change of variable to obtain
	\begin{align*}
		\|\tilde{\rho}^\kappa-\bar{\rho}^\kappa\|_{L^2_r(1,\infty)}^2&=\int_1^\infty \left|\tilde{\rho}^\kappa(r)-\bar{\rho}\left(\frac{r-1}{\sqrt{\kappa}}\right)\right|^2 r^{n-1}\,dr\\
		&=\sqrt{\kappa}\int_0^\infty \left|\tilde{\rho}^\kappa(1+\sqrt{\kappa}y)-\bar{\rho}(y)\right|^2(1+\sqrt{\kappa}y)^{n-1}\,dy.
	\end{align*}
	Thus, in terms of the variable $y$, we can restate Theorem \ref{thm:vanishing_capillarity} (2) as
	\[\|\tilde{\rho}^{\kappa}(1+\sqrt{\kappa} \cdot)- \bar{\rho}\|_{L^2_y(0,\infty)}\le C\kappa^{\frac{1}{2}},\]
	where $L^2_y$-norm defined on $[0,\infty)$ is given as
	\[\|f\|_{L^2_y(0,\infty)}:=\left(\int_0^\infty |f(y)|^2(1+\sqrt{\kappa}y)^{n-1}\,dy\right)^{1/2}.\]
	Similarly, the estimate on the derivative can be rewritten as
	\begin{align*}
		\|(\tilde{\rho}^\kappa-\bar{\rho}^\kappa)_r\|^2_{L^2_r(1,\infty)} &=\int_1^\infty \left|\left(\tilde{\rho}^\kappa(r)-\bar{\rho}\left(\frac{r-1}{\sqrt{\kappa}}\right)\right)_r\right|^2 r^{n-1}\,dr\\
		&=\frac{1}{\sqrt{\kappa}}\int_0^\infty|(\tilde{\rho}^\kappa(1+\sqrt{\kappa}y)-\bar{\rho}(y))_y|^2 (1+\sqrt{\kappa}y)^{n-1}\,dy,
	\end{align*} 
	which implies
	\[\|(\tilde{\rho}^\kappa(1+\sqrt{\kappa}\cdot)-\bar{\rho})_y\|_{L^2_y(0,\infty)}\le C\kappa^{1/2}.\]
	
\end{remark}

The rest of the paper is organized as follows. In Section \ref{sec:impermeable}, we first deal with the impermeable wall problem and prove Theorem \ref{thm:impermeable}. Then, we prove the existence of the stationary solution for the inflow and outflow problems in Section \ref{sec:inflow_outflow}, providing the proof of Theorem \ref{thm:inflow_outflow}. Section \ref{sec:vanishing} presents the proof of Theorem \ref{thm:vanishing_capillarity}. In Section \ref{sec:numeric}, we numerically validate that the convergence rate of the stationary solution for the impermeable wall problem obtained in the previous section is indeed optimal. In Appendix \ref{sec:app}, we collect some background for the spherically symmetric Helmholtz equation, and modified Bessel functions that plays crucial role in the analysis.

\section{Existence of stationary solution for the impermeable wall problem}\label{sec:impermeable}
\setcounter{equation}{0}
In this section, we prove Theorem \ref{thm:impermeable}, the existence of the stationary solution for the impermeable wall problem \eqref{NSK_stationary_impermeable}--\eqref{boundary_impermeable}. For the polytropic pressure $P(\rho) = \rho^\gamma$, with $\gamma\ge1$, the equation \eqref{NSK_stationary_impermeable} becomes
\begin{equation}\label{impermeable_1} h(\tilde{\rho})_r=\kappa\left(\tilde{\rho}_{rr}+\frac{n-1}{r}\tilde{\rho}_r\right)_r,
\end{equation}
where
\begin{equation}\label{eq:h}
	h(\rho) := \int^\rho \frac{P'(\sigma)}{\sigma}\,d\sigma=\begin{cases}
	\frac{\gamma}{\gamma-1}\rho^{\gamma-1}\quad&\mbox{if}\quad \gamma>1,\\
	\log\rho \quad&\mbox{if}\quad\gamma=1.
\end{cases}
\end{equation}
Integrating \eqref{impermeable_1} over $(r,\infty)$, we observe that $\tilde{\rho}$ satisfies the following nonlinear equation
\begin{equation}\label{impermeable_2}
\tilde{\rho}_{rr}+\frac{n-1}{r}\tilde{\rho}_r=\frac{1}{\kappa}\left(h(\tilde{\rho})-h(\rho_+)\right),
\end{equation}
subject to the boundary conditions
\[\lim_{r\to\infty}\tilde{\rho}(r) = \rho_+,\quad \tilde{\rho}_r(1) = \rho_b.\]
Now, let $\phi(r) = \tilde{\rho}-\rho_+$ be a perturbation. Then, it follows from \eqref{impermeable_2} that $\phi$ satisfies
\begin{equation}\label{eq:phi_impermeable}
	\phi_{rr} + \frac{n-1}{r}\phi_{r} - \frac{h'(\rho_+)}{\kappa}\phi=\frac{1}{\kappa}\left(h(\phi+\rho_+)-h(\rho_+)-h'(\rho_+)\phi\right)
\end{equation}
and
\begin{equation}\label{eq:phi_boundary}
	\lim_{r\to\infty}\phi(r)= 0,\quad\phi_r(1)=\rho_b.
\end{equation}
Thus, in order  to prove Theorem \ref{thm:impermeable}, it suffice to prove the following proposition, which is the goal of this section.

\begin{proposition}\label{prop:impermeable}
	Let $\sigma$ be any positive constant satisfying $0<\sigma<\sqrt{\frac{h'(\rho_+)}{\kappa}}$. Then, there exists a positive constant $\delta$ and $C=C(n,\kappa,\rho_+)$ such that the following holds. If $|\rho_b|<\delta$, there exists a unique smooth solution $\phi$ to \eqref{eq:phi_impermeable}--\eqref{eq:phi_boundary} satisfying
	\[|\phi(r)|\le C|\rho_b| e^{-\sigma r},\quad r\ge 1.\]
\end{proposition}

\subsection{Linearized system}
To prove Proposition \ref{prop:impermeable}, we first consider the non-homogeneous linear equation
\begin{equation}\label{linear}
	\phi_{rr} +\frac{n-1}{r}\phi_{r} -\alpha^2\phi =F(r),\quad \alpha:=\sqrt{\frac{h'(\rho_+)}{\kappa}},
\end{equation}
subject to \eqref{eq:phi_boundary}, where $F:[1,\infty)\to\R$ is a given function.
When $F\equiv0$, the equation \eqref{linear} corresponds to the spherically symmetric form of the Helmholtz equation, whose general solution is given by a linear combination of the (weighted) modified Bessel functions:
\begin{equation}\label{basis_solutions}
	(\alpha r)^{-\nu}I_{\nu}(\alpha r),\quad\mbox{and}\quad (\alpha r)^{-\nu}K_{\nu}(\alpha r),\quad \nu := \frac{n-2}{2}.
\end{equation}
Here, $I_\nu$ and $K_\nu$ are the modified Bessel functions of first and second kinds with order $\nu$. We refer to Appendix \ref{sec:app} for the spherically symmetric Helmholtz equation and properties of the modified Bessel function.

To match the boundary condition $\phi_r(1)=\rho_b$ in \eqref{eq:phi_boundary}, we define a lifting function
\begin{equation}\label{lifting}
	\phi_b(r) := -\frac{\rho_b}{\alpha^{1-\nu}K_{\nu+1}(\alpha)}(\alpha r)^{-\nu}K_{\nu}(\alpha r)=-\frac{\rho_b}{\alpha K_{\nu+1}(\alpha)}r^{-\nu}K_\nu(\alpha r).
\end{equation}
Then, using the properties \eqref{property_Bessel_1} and \eqref{property_Bessel_2}, one can observe that $\phi_b$ matches the boundary conditions \eqref{eq:phi_boundary}:
\[ \lim_{r\to\infty}\phi_b(r) = 0,\quad \phi_b'(1)=-\frac{\rho_b}{\alpha^{1-\nu}K_{\nu+1}(\alpha)}\alpha \left(-\alpha^{-\nu}K_{\nu+1}(\alpha)\right)=\rho_b,\quad ':=\frac{d}{dr}.\]
Furthermore, since it is a constant multiple of $(\alpha r)^{-\nu}K_\nu(\alpha r)$, it satisfies the same spherically symmetric Helmholtz equation:
\[(\phi_b)_{rr}+\frac{n-1}{r}(\phi_b)_r-\alpha^2\phi_b = 0.\]

Thus, in order to solve \eqref{linear} subject the boundary condition \eqref{eq:phi_boundary}, we first find the solution to \eqref{linear} with the homogeneous boundary conditions 
\[\lim_{r\to\infty}\phi(r) = 0,\quad \phi_r(1)=0,\]
and then just add the lifting function $\rho_b$. The solution to the non-homogeneous linear equation \eqref{linear} with the 
homogeneous boundary condition can be written as
\begin{equation}\label{expression_phi}
	\phi(r)=\int_1^\infty G(r,s)F(s)s^{n-1}\,ds,
\end{equation}
where the Green function $G(r,s)$ is given by
\[G(r,s)=\frac{1}{s^{n-1}W(s)}\begin{cases}
	\phi_+(s)\phi_-(r),\quad\mbox{if}\quad 1\le s\le r,\\
	\phi_+(r)\phi_-(s),\quad\mbox{if}\quad r\le s<\infty.
\end{cases}\]
Here, $\phi_\pm$ are the solutions to the homogeneous linear equation
\[\phi_{rr}+\frac{n-1}{r}\phi_r -\alpha^2\phi=0\] 
that satisfying one of the boundary conditions:
\[\phi_+'(1) = 0,\quad \lim_{r\to\infty}\phi_-(r)=0,\]
and $W(s):=\phi_+(s)\phi_-'(s)-\phi_+'(s)\phi_-(s)$ is a Wronskian. The derivation of \eqref{expression_phi} can be easily followed by applying the method of variation of parameters in standard ODE theory. For the completeness of the paper, we present the details for the derivation of \eqref{expression_phi} in Appendix \ref{sec:app}. By using the basis solutions in \eqref{basis_solutions} and \eqref{property_Bessel_1}, one can determine $\phi_\pm(r)$ up to constant multiplications as
\[\phi_+(r) = K_{\nu+1}(\alpha)(\alpha r)^{-\nu}I_\nu(\alpha r)+I_{\nu+1}(\alpha)(\alpha r)^{-\nu}K_\nu(\alpha r),\quad \phi_-(r) = (\alpha r)^{-\nu}K_\nu(\alpha r),\]
and therefore
\[\phi_+'(r)=\alpha(\alpha r)^{-\nu}(K_{\nu+1}(\alpha)I_{\nu+1}(\alpha r)-I_{\nu+1}(\alpha)K_{\nu+1}(\alpha r)),\quad\phi_-'(r) = -\alpha(\alpha r)^{-\nu}K_{\nu+1}(\alpha r).\]
Thus, the corresponding Wronskian can be computed as
\begin{align*}
	W(r) &=\phi_+(r)\phi_-'(r) - \phi_+'(r)\phi_-(r)\\
	&=\alpha(\alpha r)^{-2\nu}\Big[-\left(K_{\nu+1}(\alpha)I_{\nu}(\alpha r)+I_{\nu+1}(\alpha)K_{\nu}(\alpha r)\right)K_{\nu+1}(\alpha r)\\
	&\hspace{2.5cm}-(K_{\nu+1}(\alpha)I_{\nu+1}(\alpha r)-I_{\nu+1}(\alpha)K_{\nu+1}(\alpha r))K_{\nu}(\alpha r)\Big]\\
	&=-\alpha(\alpha r)^{-2\nu}K_{\nu+1}(\alpha)(I_{\nu}(\alpha r)K_{\nu+1}(\alpha r)+I_{\nu+1}(\alpha r)K_{\nu}(\alpha r))\\
	&=-\alpha(\alpha r)^{-2\nu}K_{\nu+1}(\alpha) \left(\frac{1}{\alpha r}\right)=-\frac{\alpha K_{\nu+1}(\alpha)}{(\alpha r)^{2\nu+1}}=-\frac{\alpha K_{\nu+1}(\alpha)}{(\alpha r)^{n-1}}. 
\end{align*}
Hence, the Green function can be written as
\[G(r,s)=-\frac{1}{K_{\nu+1}(\alpha)}\frac{1}{ r^\nu s^\nu}\begin{cases}
	(K_{\nu+1}(\alpha)I_{\nu}(\alpha s)+I_{\nu+1}(\alpha)K_\nu(\alpha s))K_{\nu}(\alpha r),\quad\mbox{if}\quad 1\le s\le r,\\
	(K_{\nu+1}(\alpha)I_{\nu}(\alpha r)+I_{\nu+1}(\alpha)K_\nu(\alpha r))K_{\nu}(\alpha s),\quad\mbox{if}\quad r\le s<\infty.
\end{cases}\]

The following pointwise bound on the Green function is useful in the later analysis.

\begin{lemma}\label{lem:Green}
	There exists a positive constant $C_*=C_*(n)$ depending only on $n$ such that
	\[|G(r,s)|\le \frac{C_*e^{-\alpha|r-s|}}{\alpha r^{(n-1)/2}s^{(n-1)/2}},\quad |\pa_rG(r,s)|\le \frac{C_*e^{-\alpha|r-s|}}{r^{(n-1)/2}s^{(n-1)/2}}.\]
\end{lemma}
\begin{proof}
	It follows from the bounds on the modified Bessel functions \eqref{property_Bessel_2} that there exists a constant $C$, depending only on $\nu=\frac{n-2}{2}$ that
	\[\frac{C^{-1}e^{\alpha s}}{\sqrt{\alpha s}}\le I_{\nu}(\alpha s)\le \frac{Ce^{\alpha s}}{\sqrt{\alpha s}},\quad \frac{C^{-1}e^{-\alpha s}}{\sqrt{\alpha s}}\le K_{\nu}(\alpha s)\le \frac{Ce^{-\alpha s}}{\sqrt{\alpha s}},\quad s\ge 1.\]
	Thus,
	the Green function can be estimated as
	\begin{align*}
		|G(r,s)|&\le \frac{C}{K_{\nu+1}(\alpha)r^\nu s^\nu}\begin{cases}
			\left(K_{\nu+1}(\alpha)\frac{e^{\alpha s}}{\sqrt{\alpha s}}+I_{\nu+1}(\alpha)\frac{e^{-\alpha s}}{\sqrt{\alpha s}}\right)\frac{e^{-\alpha r}}{\sqrt{\alpha r}}\quad\mbox{if}\quad 1\le s\le r\\
			\left(K_{\nu+1}(\alpha)\frac{e^{\alpha r}}{\sqrt{\alpha r}}+I_{\nu+1}(\alpha)\frac{e^{-\alpha r}}{\sqrt{\alpha r}}\right)\frac{e^{-\alpha s}}{\sqrt{\alpha s}},\quad r\le s<+\infty
			\end{cases}\\
			&\le \frac{C}{\alpha r^{\nu+1/2} s^{\nu+1/2}}\left(e^{-\alpha|r-s|}+\frac{I_{\nu+1}(\alpha)}{K_{\nu+1}(\alpha)}e^{-\alpha(r+s)}\right)\\
			&=\frac{C}{\alpha r^{(n-1)/2}s^{(n-1)/2}}\left(e^{-\alpha|r-s|}+e^{-\alpha(r+s-2)}\right)\\
			&\le \frac{Ce^{-\alpha|r-s|}}{\alpha r^{(n-1)/2}s^{(n-1)/2}},
	\end{align*}
	where the last inequality comes from $|r-s|\le r+s-2$ for any $r,s\ge 1$.
	Furthermore, using \eqref{property_Bessel_1} we compute the derivative of the Green function as, if $1\le s<r$,
	\begin{align*}
		|\pa_rG(r,s)| &= \left|\frac{\alpha}{K_{\nu+1}(\alpha)}\frac{1}{r^\nu s^\nu}(K_{\nu+1}(\alpha)I_\nu(\alpha s)+I_{\nu+1}(\alpha)K_{\nu}(\alpha s))K_{\nu+1}(\alpha r)\right|\\
		&\le \frac{Ce^{-\alpha|r-s|}}{r^{(n-1)/2}s^{(n-1)/2}},
	\end{align*}
	and if $r\le s<\infty$,
	\begin{align*}
		|\pa_rG(r,s)| &= \left|-\frac{\alpha}{K_{\nu+1}(\alpha)}\frac{1}{r^\nu s^\nu}\left(K_{\nu+1}(\alpha)I_{\nu+1}(\alpha r)-I_{\nu+1}(\alpha)K_{\nu+1}(\alpha r)\right)K_\nu(\alpha s)\right|\\
		&\le \frac{Ce^{-\alpha|r-s|}}{r^{(n-1)/2}s^{(n-1)/2}}.
	\end{align*}
	Thus, the derivative has the desired decay.
\end{proof}

\subsection{Proof of Proposition \ref{prop:impermeable}}

Based on the analysis on the expression \eqref{expression_phi} for the linearized system \eqref{linear}, the existence of the solution to \eqref{eq:phi_impermeable}--\eqref{eq:phi_boundary} is reduced to the existence of the following integral equation 
\begin{equation}\label{eq:impermeable-integral}
	\phi(r) = \phi_b(r)+\frac{1}{\kappa}\int_1^\infty G(r,s)N(\phi(s))s^{n-1}\,ds,
\end{equation}
where the nonlinear part $N(\phi)$ is defined as
\begin{equation}\label{def:N_impermeable}
	N(\phi):=h(\phi+\rho_+)-h(\rho_+)-h'(\rho_+)\phi.
\end{equation}
To show the existence of the solution to \eqref{eq:impermeable-integral}, we use the fixed point argument. To this end, for any fixed $\sigma\in(0,\alpha)$, we define the Banach space $X_\sigma$ as
\[X_\sigma:=\left\{\phi\in C([1,\infty))~\Big|~\sup_{r\ge 1}\,e^{\sigma r}|\phi(r)|<+\infty\right\},\]
together with the norm
\[\|\phi\|_{X_\sigma}:=\sup_{r\ge 1}\,e^{\sigma r}|\phi(r)|.\]
Define the ball of radius $R$ in $X_\sigma$ as
\[B_R:=\{\phi\in X_\sigma~|~\|\phi\|_{X_\sigma}<R\}.\]
We start with the boundedness and Lipschitz continuity of the nonlinear functional $N(\phi)$.

\begin{lemma}\label{lem:Nphi}
	Let $\phi,\psi\in B_R$ with $R<\frac{\rho_+}{2}$. Then, there exists a constant $C_0$, which only depends on $\rho_+$, and is independent of $R$, such that
	\[\|N(\phi)\|_{X_\sigma}\le C_0\|\phi\|^2_{X_\sigma}\le C_0R^2,\quad \|N(\phi)-N(\psi)\|_{X_\sigma}\le C_0R\|\phi-\psi\|_{X_\sigma}.\]
\end{lemma}
\begin{proof}
	Using the Taylor theorem, $N(\phi)$ can be written as
	\[N(\phi) = h(\phi+\rho_+)-h(\rho_+)-h'(\rho_+)\phi=\frac{h''(\rho_*)}{2}\phi^2,\]
	for some $\rho_*$ between $\rho_+$ and $\phi+\rho_+$. Specifically, depending on the value of $\gamma$, $N(\phi)$ can be represented as
	\[N(\phi) = \begin{cases}
		\frac{\gamma(\gamma-2)\rho_*^{\gamma-3}}{2}\phi^2,&\quad\mbox{if}\quad \gamma>1,\\
		-\frac{1}{2\rho_*^2}\phi^2,&\quad\mbox{if}\quad \gamma=1.
	\end{cases}\]
	If $\phi\in B_R$, we have for $r\ge1$ that  
	\[|\phi(r)|\le e^{\sigma r}|\phi(r)|\le \|\phi\|_{X_\sigma}\le R.\]
	Thus, if $R<\frac{\rho_+}{2}$, we have $\frac{\rho_+}{2}<\phi(r)+\rho_+<\frac{3\rho_+}{2}$ for all $r\ge1$. This implies that $\frac{\rho_+}{2}<\rho_*<\frac{3\rho_+}{2}$, and therefore $|N(\phi)|\le \frac{|h''(\rho_*)|}{2}|\phi|^2\le C|\phi|^2$ for some constant $C$ that only depends on $\rho_+$. Hence, 
	\[\|N(\phi)\|_{X_\sigma}=\sup_{r\ge 1}\,e^{\sigma r}|N(\phi(r))|\le \sup_{r\ge 1}\,Ce^{\sigma r}|\phi(r)|^2 \le C\|\phi\|^2_{X_\sigma}\le CR^2.\]
	To derive the second inequality, we first note that
	\[N(\phi)-N(\psi) = \int_0^1 \frac{d}{d\tau}N(\psi+\tau(\phi-\psi))\,d\tau=(\phi-\psi)\int_0^1 N'(\tau\phi+(1-\tau)\psi)\,d\tau,\]
	where 
	\[N'(\phi) = h'(\phi+\rho_+)-h'(\rho_+)=h''(\rho_{**})\phi\]
	for some $\rho_{**}$ between $\rho_+$ and $\phi+\rho_+$. Then, since $\phi,\psi\in B_R$, we have
	\begin{align}
		\begin{aligned}\label{Nphipsi}
		|N(\phi)-N(\psi)|&\le |\phi-\psi|\int_0^1 |N'(t\phi+(1-t)\psi)|\,dt\\
		&\le C|\phi-\psi|\int_0^1 |t\phi+(1-t)\psi|\,dt\le C(|\phi|+|\psi|)|\phi-\psi|\le CR|\phi-\psi|.
		\end{aligned}
	\end{align}
	Here, the constant $C$ only depends on $\rho_+$ as before. After multiplying $e^{\sigma r}$ on both sides of the inequality \eqref{Nphipsi}, and taking the supremum over $r\ge1$, the desired inequality for $\|N(\phi)-N(\psi)\|_{X_\sigma}$ follows.
\end{proof}
Now, for any $\phi\in X_\sigma$, we define the map $\mathcal{T}$ as
\[\mathcal{T}[\phi](r):=\phi_b(r)+\frac{1}{\kappa}\int_1^\infty G(r,s)N(\phi(s)) s^{n-1}\,ds.\]
Then, we show that for sufficiently small $\rho_b$, we can choose $R>0$ so that the map $\mathcal{T}$ is a contraction map from $B_R$ to $B_R$.

\begin{lemma}\label{lem:contraction_impermeable}
	There exists $\delta>0$ such that the following holds. If $|\rho_b|<\delta$, then, there exists $0<R<\frac{\rho_+}{2}$ such that  $\mathcal{T}$ maps $B_R$ into $B_R$. Furthermore, for any $\phi,\psi\in B_R$,
	\[\|\mathcal{T}[\phi]-\mathcal{T}[\psi]\|_{X_\sigma}\le \frac{1}{2}\|\phi-\psi\|_{X_\sigma}.\] 
\end{lemma}

\begin{proof}
	We temporally choose $R<\frac{\rho_+}{2}$. Then, it follows from the definition of $\mathcal{T}$, $\phi_b$, Lemma \ref{lem:Green}, and Lemma \ref{lem:Nphi} that there exist positive constants $C_1$ and $C_2$ that only depend on $\alpha,\sigma$, and $\rho_+$ such that
	\begin{align*}
		e^{\sigma r}|\mathcal{T}[\phi](r)|&\le e^{\sigma r}|\phi_b(r)|+\frac{e^{\sigma r}}{\kappa}\left|\int_1^\infty G(r,s)N(\phi(s))s^{n-1}\,ds\right|\\
		&\le \frac{|\rho_b|e^{\sigma r}}{\alpha K_{\nu+1}(\alpha)}r^{-\nu} K_{\nu}(\alpha r) + \frac{C_*e^{\sigma r}}{\alpha\kappa}\int_1^\infty \frac{e^{-\alpha|r-s|}}{r^{(n-1)/2}s^{(n-1)/2}}|N\phi(s)|s^{n-1}\,ds\\
		&\le \frac{C_1|\rho_b|e^{-(\alpha-\sigma)r}}{r^{(n-1)/2}} +\frac{C_*\|N(\phi)\|_{X_\sigma}e^{\sigma r}}{\kappa r^{(n-1)/2}}\int_1^\infty e^{-\alpha|r-s|}e^{-\sigma s}s^{(n-1)/2}\,ds\\
		&\le C_1|\rho_b| + \frac{C_*C_0R^2e^{\sigma r}}{\kappa r^{(n-1)/2}}\left(\int_1^r e^{-\alpha(r-s)}e^{-\sigma s}s^{(n-1)/2}\,ds+\int_r^\infty e^{-\alpha(s-r)}e^{-\sigma s}s^{(n-1)/2}\,ds\right)\\
		&\le C_1|\rho_b| + \frac{C_*C_0R^2e^{\sigma r}}{ r^{(n-1)/2}}\left(r^{(n-1)/2}e^{-\alpha r}\int_1^r e^{(\alpha-\sigma)s}\,ds+e^{\alpha r}\int_r^\infty e^{-(\alpha+\sigma)s}s^{(n-1)/2}\,ds\right)\\
		&\le C_1|\rho_b| + \frac{C_*C_0R^2e^{\sigma r}}{ r^{(n-1)/2}}\left(\frac{r^{(n-1)/2}e^{-\sigma r}}{\alpha-\sigma} + C(\alpha,\sigma)e^{\alpha r}r^{(n-1)/2}e^{-(\alpha+\sigma)r}\right)\\
		&\le C_1|\rho_b| + C_2R^2.
	\end{align*}
	Here, we used the following estimate in the last inequality:
	\begin{align*}
		&e^{\alpha r}\int_r^\infty e^{-(\alpha+\sigma)s}s^{(n-1)/2}\,ds\\
		&=e^{\alpha r}\int_0^\infty e^{-(\alpha+\sigma)(r+t)}(r+t)^{(n-1)/2}\,dt\\
		&=e^{\alpha r}e^{-(\alpha+\sigma)r}r^{(n-1)/2}\int_0^\infty e^{-(\alpha+\sigma)t}\left(1+\frac{t}{r}\right)^{(n-1)/2}\,dt\\
		&\le e^{-\sigma r}r^{(n-1)/2}\int_0^\infty e^{-(\alpha+\sigma)t}\left(1+t\right)^{(n-1)/2}\,dt = C(\alpha,\sigma)e^{-\sigma r}r^{(n-1)/2}. 
	\end{align*}
	Therefore, taking supremum over $r\ge1$, we get
	\[\|\mathcal{T}[\phi]\|_{X_\sigma}\le C_1|\rho_b| + C_2R^2.\]
	Now, if $|\rho_b|$ is sufficiently small so that $|\rho_b|<\min\{\frac{\rho_+}{4C_1},\frac{1}{4C_1C_2}\}$, 
	we choose $R=2C_1|\rho_b|<\frac{\rho_+}{2}$. Then, for such $R$,
	\[\|\mathcal{T}[\phi]\|_{X_\sigma}\le C_1|\rho_b| + C_2R^2=C_1|\rho_b| + 4C_1^2C_2|\rho_b|^2<2C_1|\rho_b|=R.\]
	Hence, $\mathcal{T}$ maps $B_R$ into $B_R$. To show the contraction property of $\mathcal{T}$, we first note that for any $\phi,\psi\in B_R$,
	\[\mathcal{T}[\phi]-\mathcal{T}[\psi]=\frac{1}{\kappa}\int_1^\infty G(r,s)(N(\phi)-N(\psi))s^{(n-1)/2}\,ds.\]
	Then, using Lemma \ref{lem:Nphi}, we follow similar estimate as above to conclude that there exists a positive constant $C_3$ that only depends on $\alpha,\sigma$ and $\rho_+$ such that
	\[\|\mathcal{T}[\phi]-\mathcal{T}[\psi]\|_{X_\sigma}\le C_3R\|\phi-\psi\|_{X_\sigma}=2C_1C_3|\rho_b|\|\phi-\psi\|_{X_\sigma}.\]
	Thus, if $|\rho_b|$ is further sufficiently small so that $|\rho_b|<\frac{1}{4C_1C_3}$, we have $\|\mathcal{T}[\phi]-\mathcal{T}[\psi]\|_{X_\sigma}\le \frac{1}{2}\|\phi-\psi\|_{X_\sigma}$.
To sum up, if we choose $\delta$ as
\begin{equation}\label{smallness_impermeable}
	\delta =\frac{1}{4C_1} \min\left\{\rho_+,\frac{1}{C_2},\frac{1}{C_3}\right\},
\end{equation}
then for any $\rho_b$ satisfying $|\rho_b|<\delta$, the map $\mathcal{T}$ is a contraction map from $B_R$ to $B_R$ for $R=2C_1|\rho_b|$. 
\end{proof}

Therefore, if $|\rho_b|<\delta$ for $\delta$ defined in \eqref{smallness_impermeable}, there exists a unique fixed point $\phi\in B_{2C_1|\rho_b|}$ of $\mathcal{T}$, which satisfies the integral equation \eqref{eq:impermeable-integral}. This proves the existence of the solution to \eqref{eq:phi_impermeable}--\eqref{eq:phi_boundary}. Furthermore, since $\phi\in B_{2C_1|\rho_b|}$, we have
\[e^{\sigma r}|\phi(r)|\le \|\phi\|_{X_\sigma} \le 2C_1|\rho_b|,\quad r\ge 1,\]
which is the desired decay estimate of $\phi$. This completes the proof of Proposition \ref{prop:impermeable}.

\begin{remark}
When $n=3$, the modified Bessel function $I_\nu(z)$ and $K_\nu(z)$ with $\nu = \frac{n-2}{2}=\frac{1}{2}$ can be explicitly written as
\[I_{1/2}(z)=\sqrt{\frac{2}{\pi z}}\sinh z,\quad K_{1/2}(z) = \sqrt{\frac{\pi}{2z}}e^{-z}.\] 
This explicit formulae of the modified Bessel functions simplify the expression of the Green function and the lifting function. Precisely, the general solution to the modified Helmholtz equations are the linear combinations of $\frac{e^{\alpha r}}{r}$ and $\frac{e^{-\alpha r}}{r}$. In this case, we have
\[\phi_+(r) = e^{-\alpha}(\alpha+1)\frac{e^{\alpha r}}{r}+e^\alpha(\alpha-1)\frac{e^{-\alpha r}}{r},\quad \phi_-(r) = \frac{e^{-\alpha r}}{r},\]
and the Wronskian can be computed as
\[W(r) = \phi_+(r)\phi'_-(r)-\phi'_+(r)\phi_-(r)=-\frac{2\alpha(\alpha+1)e^{-\alpha}}{r^2}.\]
Hence, the Green function in three-dimensional case becomes
\[G(r,s) =-\frac{1}{2\alpha}\frac{e^{-\alpha|r-s|}}{rs}-\frac{\alpha-1}{2\alpha(\alpha+1)}\frac{e^{-\alpha(r+s-2)}}{rs}.\]
Furthermore, the lifting function can be simplified as
\[\phi_b(r) = -\frac{e^\alpha\rho_b}{(\alpha+1)}\frac{e^{-\alpha r}}{r}.\]
\end{remark}

\section{Existence of stationary solution for the inflow/outflow problems}\label{sec:inflow_outflow}
\setcounter{equation}{0}
In this section, we deal with the existence of the stationary solution $(\tilde{\rho},\tilde{u})$ to the inflow/outflow problems, proving Theorem \ref{thm:inflow_outflow}. We first divide \eqref{NSK_stationary}$_2$ by $\tilde{\rho}$, and then integrate it over $(r,\infty)$ to get
\begin{equation}\label{inflow_1}
	-\frac{\tilde{u}^2}{2}+h(\rho_+)-h(\tilde{\rho})=\mu\int_r^\infty \frac{1}{\tilde{\rho}}\left(\frac{(r^{n-1}\tilde{u})_r}{r^{n-1}}\right)_r\,dr-\kappa\left(\tilde{\rho}_{rr}+\frac{n-1}{r}\tilde{\rho}_r\right).
\end{equation}
As in \eqref{tildeu}, $\tilde{u}$ can be represented in terms of $\tilde{\rho}$ as  
\[\tilde{u}(r) = \frac{\tilde{\rho}(1)u_-}{\tilde{\rho}(r)r^{n-1}}.\]
After substituting it into \eqref{inflow_1} and taking integration by parts, we derive the following integro-differential equation for $\tilde{\rho}$:
\begin{equation}\label{inflow_2}
	\kappa\left(\tilde{\rho}_{rr}+\frac{n-1}{r}\tilde{\rho}_r\right) =\frac{\mu \tilde{\rho}(1)u_-}{r^{n-1}\tilde{\rho}^3}\tilde{\rho}_r+ h(\tilde{\rho})-h(\rho_+)+\frac{\tilde{\rho}(1)^2u_-^2}{2r^{2(n-1)}\tilde{\rho}^2}-\mu \tilde{\rho}(1)u_-\int_r^{\infty}\frac{\tilde{\rho}_r(s)^2}{s^{n-1}\tilde{\rho}(s)^4}\,ds,
\end{equation}
subject to the boundary conditions
\begin{equation*}
	\lim_{r\to\infty}\tilde{\rho}(r)=\rho_+,\quad \tilde{\rho}_r(1)=\rho_b.
\end{equation*}
Note that, when $u_-=0$, the equation \eqref{inflow_2} reduces to \eqref{impermeable_2} which is the equation for the impermeable wall problem. Now, we rewrite the equation \eqref{inflow_2} in terms of the perturbation $\phi:=\tilde{\rho}-\rho_+$ as

\begin{align}
	\begin{aligned}\label{eq:phi_inflow_outflow}
		&\phi_{rr} + \frac{n-1}{r}\phi_r -\frac{h'(\rho_+)}{\kappa}\phi\\
		&=\frac{1}{\kappa}\Bigg(\frac{\mu (\phi(1)+\rho_+)u_-\phi_r}{r^{n-1}(\phi+\rho_+)^3}+(h(\phi+\rho_+)-h(\rho_+)-h'(\rho_+)\phi)\\
		&\hspace{0.5cm}+\frac{(\phi(1)+\rho_+)^2u_-^2}{2r^{2(n-1)}(\phi+\rho_+)^2}-\mu (\phi(1)+\rho_+)u_-\int_r^\infty \frac{\phi_r^2(s)}{s^{n-1}(\phi(s)+\rho_+)^4}\,ds\Bigg)=:\frac{1}{\kappa}\left(S(r) + N(\phi)\right),
	\end{aligned}
\end{align}
where $S(r)$ is the source term that is independent of $\phi$, which is
\begin{equation}\label{def:S_inflow_outflow}
	S(r):=\frac{u_-^2}{2r^{2(n-1)}},
\end{equation}
while $N(\phi)$ is the nonlinear terms defined as
\begin{align}
	\begin{aligned}\label{def:N_inflow_outflow}
		N(\phi) &= \frac{\mu (\phi(1)+\rho_+)u_-\phi_r}{r^{n-1}(\phi+\rho_+)^3}+(h(\phi+\rho_+)-h(\rho_+)-h'(\rho_+)\phi)\\
		&\quad + \frac{u_-^2}{2r^{2(n-1)}}\left(\frac{(\phi(1)+\rho_+)^2}{(\phi+\rho_+)^2}-1\right)-\mu (\phi(1)+\rho_+)u_-\int_r^\infty \frac{\phi_r^2(s)}{s^{n-1}(\phi(s)+\rho_+)^4}\,ds.
	\end{aligned}
\end{align}
Again, the boundary conditions for $\phi$ are imposed as
\begin{equation}\label{eq:phi_boundary_inflow_outflow}
	\lim_{r\to\infty}\phi(r)=0,\quad \phi_r(1)=\rho_b.
\end{equation}

\begin{proposition}\label{prop:inflow}
	There exists a positive constant $\delta$ and $C=C(n,\kappa\rho_+)$ such that the following holds. If $|\rho_b|,|u_-|<\delta$, then there exists a unique smooth solution $\phi$ to \eqref{eq:phi_inflow_outflow}--\eqref{eq:phi_boundary_inflow_outflow} satisfying
	\[|\phi(r)|\le C(|\rho_b|+|u_-|^2)r^{-2(n-1)},\quad |\phi_r(r)|\le C(|\rho_b|+|u_-|^2)r^{-(2n-1)},\quad r\ge1.\]
\end{proposition}

Unlike the impermeable wall problem, we cannot expect the exponential decay of $\phi$, due to the algebraic decaying source term $S(r)$. Therefore, we define a new Banach space $Y$ as
\[Y:=\left\{\phi\in C^1([1,\infty))~\Big|~\sup_{r\ge 1}\,\left(r^{2(n-1)}|\phi(r)|+r^{2n-1}|\phi_r(r)|\right)<+\infty\right\},\]
where the norm on $Y$ is defined as
\[\|\phi\|_Y:=\sup_{r\ge 1}\,\left(r^{2(n-1)}|\phi(r)|+r^{2n-1}|\phi_r(r)|\right).\]

Similar to the impermeable wall problem, we choose the same lifting function $\phi_b(r)$ defined in \eqref{lifting} to get rid of the boundary condition $\phi_r(1) = \rho_b$. Then, according to the observation in the previous section, the solution to \eqref{eq:phi_inflow_outflow} is a fixed point of the following map:
\[\mathcal{T}[\phi](r) = \phi_b(r)+\frac{1}{\kappa}\int_1^\infty G(r,s)\left(S(s)+N(\phi(s))\right)s^{n-1}\,ds.\]
As before, our goal is to show that $\mathcal{T}$ is self-containing map from suitable ball in $Y$ into itself, and it is indeed contraction map. To this end, we again define the ball of radius $R$ in $Y$ as $B_R$:
\[B_R:=\{\phi\in Y~|~\|\phi\|_Y<R\}.\]
To estimate the contraction property of $\mathcal{T}$, we first need to control the new nonlinear term $N(\phi)$.

\begin{lemma}\label{lem:Nphi_inflow}
	Let $\phi, \psi\in B_R$ with $R<\frac{\rho_+}{2}$. 
	Then, there exists a constant $C_0$, which only depends on $\rho_+$, and is independent of $R$, such that
	\[|N(\phi)|\le \frac{C_0(\mu|u_-|+R+u_-^2)}{r^{2(n-1)}}\|\phi\|_Y,\quad |N(\phi)-N(\psi)|\le \frac{C_0(\mu|u_-|+R+u_-^2)}{r^{2(n-1)}}\|\phi-\psi\|_Y.\]
\end{lemma}

\begin{proof}
	First, note that if $\phi\in B_R$ with $R<\frac{\rho_+}{2}$, we have $\frac{\rho_+}{2}\le\phi+\rho_+\le \frac{3\rho_+}{2}$.	Thus, it follows from the definition of $N(\phi)$ in \eqref{def:N_inflow_outflow} that for any $\phi\in B_R$ and $r\ge1$,
	\begin{align*}
		&|N(\phi(r))|\\
		&\quad\le C_0\left(\mu |u_-|r^{-(n-1)} |\phi_r| + |\phi|^2 + |u_-|^2r^{-2(n-1)}|\phi(1)-\phi(r)|+ \mu |u_-|\int_r^\infty \frac{|\phi^2_r(s)|}{s^{n-1}}\,ds\right)\\
		&\quad\le C_0\left(\mu |u_-| r^{-(3n-2)}\|\phi\|_Y + r^{-4(n-1)}\|\phi\|_Y^2 + |u_-|^2r^{-2(n-1)}\|\phi\|_Y + \mu |u_-|\|\phi\|_Y^2\int_r^\infty s^{-(5n-3)}\,ds\right)\\
		&\quad\le C_0(\mu |u_-|(1+R)+R+|u_-|^2)r^{-2(n-1)}\|\phi\|_Y\\
		&\quad<C_0(\mu |u_-|+R+|u_-|^2)r^{-2(n-1)}\|\phi\|_Y,
	\end{align*}
	where we use
	\[|\phi(1)-\phi(r)|\le |\phi(1)|+|\phi(r)|\le (1+r^{-2(n-1)})\|\phi\|_Y\le 2\|\phi\|_Y,\] 
	and $1+R<1+\frac{\rho_+}{2}$, and $C_0$ is the constant only depend on $\rho_+$.
	Similarly, for any $\phi,\psi\in B_R$, we have
	
	\begin{align*}
		|N(\phi(r))-N(\psi(r))|&\le \frac{\mu|u_-|}{r^{n-1}}\left|\frac{(\phi(1)+\rho_+)\phi_{r}}{(\phi(r)+\rho_+)^3}-\frac{(\psi(1)+\rho_+)\psi_{r}}{(\psi(r)+\rho_+)^3}\right|\\
		&\quad +|(h(\phi+\rho_+)-h(\rho_+)-h'(\rho_+)\phi)-(h(\psi+\rho_+)-h(\rho_+)-h'(\rho_+)\psi)|\\
		&\quad + \frac{u_-^2}{2r^{2(n-1)}}\left|\frac{(\phi(1)+\rho_+)^2}{(\phi+\rho_+)^2}-\frac{(\psi(1)+\rho_+)^2}{(\psi+\rho_+)^2}\right|\\
		&\quad +\mu|u_-|\int_r^{\infty}\left|\frac{(\phi(1)+\rho_+)\phi_{r}^2(s)}{(\phi(s)+\rho_+)^4}-\frac{(\psi(1)+\rho_+)\psi_{r}^2(s)}{(\psi(s)+\rho_+)^4}\right|\frac{1}{s^{n-1}}\,ds\\
		&=\sum_{i=1}^4I_{i}.
	\end{align*}
	We estimate $I_1$ as
	\begin{align*}
		I_1&\le \frac{C_0\mu|u_-|}{r^{n-1}}\left(\left|\phi(1)-\psi(1)\right||\phi_{r}|+|\phi_{r}-\psi_{r}|+|\psi_{r}||\phi(r)-\psi(r)|\right)\\
		&\le \frac{C_0\mu|u_-|}{r^{n-1}}\left(\|\phi-\psi\|_Y\|\phi\|_Yr^{-(2n-1)}+\|\phi-\psi\|_Yr^{-(2n-1)}+\|\psi\|_Yr^{-(4n-3)}\|\phi-\psi\|_Y\right)\\
		&\le C_0\mu|u_-|r^{-(3n-2)}\|\phi-\psi\|_Y.
	\end{align*}
	The estimate of $I_2$ is the same as in Lemma \ref{lem:Nphi}, and therefore
	\[I_2\le C_0R|\phi-\psi|\le C_0Rr^{-2(n-1)}\|\phi-\psi\|_Y.\]
	For $I_3$, we have
	\begin{align*}
		I_3&\le \frac{C_0u_-^2}{r^{2(n-1)}}\left(|\phi(1)-\psi(1)|+|\phi(r)-\psi(r)|\right)\le \frac{C_0u_-^2}{r^{2(n-1)}}\|\phi-\psi\|_Y.
	\end{align*}
	Finally, we estimate $I_4$ as
	\begin{align*}
		I_4&\le C_0\mu|u_-|\int_r^\infty \big(|\phi(1)-\psi(1)||\phi^2_{r}(s)|\\
		&\hspace{3cm}+|\phi_{r}(s)+\psi_{r}(s)||\phi_{r}(s)-\psi_{r}(s)|+|\psi_{r}(s)|^2|\phi(s)-\psi(s)|\big)\frac{1}{s^{n-1}}\,ds\\
		&\le C_0\mu|u_-|\int_r^{\infty}\left(\|\phi-\psi\|_YR^2s^{-(4n-2)}+R\|\phi-\psi\|_Ys^{-(4n-2)}+\|\psi\|^2_Y\|\phi-\psi\|_Y s^{-(6n-4)}\right)\frac{1}{s^{n-1}}\,ds\\
		&\le C_0\mu|u_-|\|\phi-\psi\|_Y\left(R+R^2\right)\int_r^\infty s^{-(5n-3)} \,ds\le C_0\mu |u_-|\|\phi-\psi\|_Y r^{-(5n-4)}.
	\end{align*}
	Thus, we combine the estimates on $I_i$ for $i=1,2,3,4$ to obtain
	\begin{align*}
		|N(\phi(r))-N(\psi(r))|\le C_0\left(\mu|u_-|+R+u_-^2\right)r^{-2(n-1)}\|\phi-\psi\|_Y.
	\end{align*}
\end{proof}

Now, as we have done in the impermeable wall problem, we will show that for sufficiently small $|\rho_b|$ and $|u_-|$, we may find $R$ so that the map $\mathcal{T}$ is a contraction mapping from $B_R$ to $B_R$. 
\begin{lemma}\label{lem:T_contraction}
There exists $\delta>0$ such that the following holds. If $|\rho_b|,|u_-|<\delta$, there exists $0<R<\frac{\rho_+}{2}$ such that $\mathcal{T}$ maps $B_R$ into $B_R$. Furthermore, for any $\phi,\psi\in B_R$,
\[\|\mathcal{T}[\phi]-\mathcal{T}[\psi]\|_{Y}\le \frac{1}{2}\|\phi-\psi\|_Y.\]
\end{lemma}

\begin{proof}
	First, for $\phi\in B_R$ with $0<R<\frac{\rho_+}{2}$, we use Lemma \ref{lem:Green} and Lemma \ref{lem:Nphi_inflow} to get
	\begin{align*}
		r^{2(n-1)}|\mathcal{T}[\phi](r)|&\le r^{2(n-1)}|\phi_b(r)| + \frac{Cr^{2(n-1)}}{\kappa}\int_1^{\infty}\frac{e^{-\alpha|r-s|}\left(|S(s)|+|N(\phi(s))|\right)}{r^{(n-1)/2}s^{(n-1)/2}}s^{n-1}\,ds\\
		&\le C|\rho_b| r^{2(n-1)}r^{-\nu}K_{\nu}(\alpha r) \\
		&\quad + \frac{Cr^{2(n-1)}}{\kappa r^{(n-1)/2}}\int_1^\infty e^{-\alpha|r-s|}\left(\frac{|u_-|^2}{2s^{2(n-1)}}+|N(\phi(s))|\right)s^{(n-1)/2}\,ds\\
		&\le C_1|\rho_b| +\frac{Cr^{3(n-1)/2}|u_-|^2}{\kappa}\int_1^\infty e^{-\alpha|r-s|}s^{-3(n-1)/2}\,ds \\
		&\quad +\frac{CC_0(\mu |u_-|+R+|u_-|^2)Rr^{3(n-1)/2}}{\kappa}\int_1^\infty e^{-\alpha|r-s|}s^{-3(n-1)/2}\,ds.
	\end{align*}
On the other hand, for any $k>0$ and $r\ge 1$, there exists a positive constant $C$ that only depend on $\alpha$ and $k$ such that 
\begin{align*}
	\int_1^\infty e^{-\alpha|r-s|}s^{-k}\,ds&=\int_1^{\max\{1,r/2\}} e^{-\alpha(r-s)}s^{-k}\,ds +\int_{\max\{1,r/2\}}^{2r}e^{-\alpha|r-s|}s^{-k}\,ds+ \int_{2r}^\infty e^{-\alpha(s-r)}s^{-k}\,ds\\
	&\le e^{-\alpha r/2}\int_{1}^{\max\{1,r/2\}}1\,ds+ (2^k+1)r^{-k}\int_{\R}e^{-\alpha|s|}\,ds+(2r)^{-k}\int_{2r}^\infty e^{-\alpha s/2}\,ds\\
	&\le e^{-\alpha r/2} + Cr^{-k}+Cr^{-k}e^{-\alpha r}\le Cr^{-k},
\end{align*}
where we use 
\[s^{-k}\le \max\{1,r/2\}^{-k}\le (2^k+1)r^{-k}\quad\mbox{for}\quad s\ge \min\{1,r/2\}.\]
Thus, we conclude that, there exists a positive constants $C_1,C_2$ and $C_3$ that are independent of $R$, $\mu$, and $u_-$ that
\[r^{2(n-1)}|\mathcal{T}[\phi](r)|\le C_1|\rho_b| + C_2|u_-|^2+C_3(\mu |u_-|+R+|u_-|^2)R.\]
Furthermore, taking derivative $\mathcal{T}[\phi](r)$ with respect to $r$, we get
\begin{align*}
	r^{2n-1}\pa_r(\mathcal{T}[\phi])(r)= r^{2n-1}\phi'_b(r)+\frac{1}{\kappa}\int_1^\infty \pa_rG(r,s)\left(S(s)+N(\phi(s))\right)s^{n-1}\,ds.
\end{align*}
Using the pointwise estimate of the derivative of the Green function in Lemma \ref{lem:Green}, we still have the same bound
\begin{align*}
	r^{2n-1}|\pa_r(\mathcal{T}[\phi])(r)|\le C_1|\rho_b| + C_2|u_-|^2+C_3(\mu |u_-|+R+|u_-|^2)R,
\end{align*}
where we may choose larger constants $C_1$, $C_2$, and $C_3$ if needed. Taking supremum over $r\ge 1$, we get
\[\|\mathcal{T}[\phi]\|_Y\le C_1|\rho_b| + C_2|u_-|^2+C_3\left(\mu |u_-|+R+|u_-|^2\right)R.\]
Thus, if $|\rho_b|$ and $|u_-|$ are sufficiently small so that
\begin{equation}\label{smallness_inflow1}
	C_1|\rho_b|+C_2|u_-|^2<\frac{\rho_+}{4},\quad \mu |u_-|+2\left(C_1|\rho_b|+C_2|u_-|^2\right)+|u_-|^2< \frac{1}{2C_3},
\end{equation}
we choose $R=2\left(C_1|\rho_b|+C_2|u_-|^2\right)<\frac{\rho_+}{2}$ to derive

\begin{align*}
	\|\mathcal{T}[\phi]\|_Y&\le \frac{R}{2}+C_3\left(\mu |u_-|+2\left(C_1|\rho_b|+C_2|u_-|^2\right)+|u_-|^2\right)R<\frac{R}{2}+\frac{R}{2}=R.
\end{align*}
Thus, $\mathcal{T}$ is a map from $B_R$ to $B_R$ itself for such $R$. To show the contraction estimate, let $\phi,\psi\in B_R$. Then, using Lemma \ref{lem:Nphi_inflow},

\begin{align*}
	|\mathcal{T}[\phi](r)-\mathcal{T}[\psi](r)|&\le \frac{1}{\kappa}\int_1^\infty G(r,s)|N(\phi)(s)-N(\psi(s))|s^{n-1}\,ds\\
	&\le \frac{CC_0(\mu |u_-|+R+|u_-|^2)\|\phi-\psi\|_Y}{\kappa r^{(n-1)/2}}\int_1^\infty e^{-\alpha|r-s|}s^{-3(n-1)/2}\,ds\\
	&\le \frac{C_4(\mu |u_-|+R+|u_-|^2)\|\phi-\psi\|_Y}{ r^{2(n-1)}},
\end{align*}
where $C_4$ is the constant that are independent of $\mu,R$, and $u_-$. Thus,
\[r^{2(n-1)}|\mathcal{T}[\phi](r)-\mathcal{T}[\psi](r)|\le C_4(\mu |u_-|+R+|u_-|^2)\|\phi-\psi\|_Y.\]
Similarly, we have
\[r^{2n-1}|\pa_r(\mathcal{T}[\phi])(r)-\pa_r(\mathcal{T}[\psi])(r)|\le C_4(\mu |u_-|+R+|u_-|^2)\|\phi-\psi\|_Y.\]
Thus, after taking supremum over $r\ge1$,
\[\|\mathcal{T}[\phi]-\mathcal{T}[\psi]\|_Y\le C_4(\mu |u_-|+R+|u_-|^2)\|\phi-\psi\|_Y.\]
Therefore, if we choose $|\rho_b|$ and $|u_-|$ further small so that
\begin{equation}\label{smallness_inflow2}
	\mu |u_-|+2\left(C_1|\rho_b|+C_2|u_-|^2\right)+|u_-|^2<\frac{1}{2C_4},
\end{equation}
the map $\mathcal{T}$ is contraction satisfying $\|\mathcal{T}[\phi]-\mathcal{T}[\psi]\|_Y<\frac12 \|\phi-\psi\|_Y$ for any $\phi,\psi\in B_{2(C_1|\rho_b|+C_2|u_-|^2)}$. 
To sum up, in order to guarantee the smallness conditions \eqref{smallness_inflow1} and \eqref{smallness_inflow2}, we choose $\rho_b$ and $u_-$ sufficiently small so that 
\[|\rho_b| < \min\left\{\frac{\rho_+}{8C_1},\frac{1}{8C_1C_3},\frac{1}{8C_1C_4}\right\},\quad |u_-|<\min\left\{1, \frac{\rho_+}{8C_2},\frac{1}{4\left(\mu+1+2C_2\right)C_3},\frac{1}{4\left(\mu+1+2C_2\right)C_4}\right\}.\]
Then, the map $\mathcal{T}$ is a self-mapping contraction from $B_R$ to $B_R$ with $R=2(C_1|\rho_b|+C_2|u_-|^2)$.
\end{proof}

Since the solution to \eqref{eq:phi_inflow_outflow} subject to the boundary conditions \eqref{eq:phi_boundary_inflow_outflow} is a fixed point of the map $\mathcal{T}$, the existence of smooth solution to \eqref{eq:phi_inflow_outflow}--\eqref{eq:phi_boundary_inflow_outflow} is guaranteed for sufficiently small $|\rho_b|$ and $|u_-|$ by a contraction property of $\mathcal{T}$ in Lemma \ref{lem:T_contraction}. Moreover since our fixed point $\phi$ belongs to the $B_{R}$ with $R=2(C_1\rho_b+C_2|u_-|^2)$, we get the desired decay bound of $\phi$. This completes the proof of Proposition \ref{prop:inflow}.

\section{Vanishing capillarity limit for impermeable wall problem}\label{sec:vanishing}
\setcounter{equation}{0}

In this section, we consider the vanishing capillarity limit of the stationary solution to the impermeable wall problem \eqref{NSK_stationary_impermeable}--\eqref{boundary_impermeable}, and completes the proof of Theorem \ref{thm:vanishing_capillarity}. 

\subsection{Fixed capillarity boundary condition}
We first consider the case when the boundary condition is fixed as $\rho_b$. Let us recall the equation for the stationary solution to the impermeable wall problem \eqref{impermeable_2}:
\[\tilde{\rho}^\kappa_{rr}+\frac{n-1}{r}\tilde{\rho}^\kappa_r=\frac{1}{\kappa}(h(\tilde{\rho}^\kappa)-h(\rho_+)).\]
Here, we use the notation $\tilde{\rho}^\kappa=\tilde{\rho}$ to express the dependency of the solution on $\kappa$. We now introduce a new variable $y=\frac{r-1}{\sqrt{\kappa}}$. Then, the equation can be represented in terms of $y$ as
\begin{equation*}
	\tilde{\rho}^\kappa_{yy}+\frac{\sqrt{\kappa}(n-1)}{1+\sqrt{\kappa}y}\tilde{\rho}^\kappa_y = h(\tilde{\rho}^\kappa)-h(\rho_+),
\end{equation*}
which formally converges to
\begin{equation}\label{eq:limit}
	\bar{\rho}_{yy} = h(\bar{\rho})-h(\rho_+)
\end{equation}
as $\kappa\to0$. Furthermore, the boundary conditions in terms of $y$ becomes
\begin{equation}\label{eq:limit_boundary}
	\bar{\rho}_y(0) = \lim_{\kappa\to0}\sqrt{\kappa}\rho_b=0,\quad \lim_{y\to\infty}\bar{\rho}(y) = \rho_+.
\end{equation}
Therefore, the solution to \eqref{eq:limit}--\eqref{eq:limit_boundary} becomes the trivial solution $\bar{\rho}(y) \equiv \rho_+$ and we can expect that the stationary solution $\tilde{\rho}^\kappa(r)$ converges to the constant $\rho_+$. Observe that the constant solution $\rho_+$ is a stationary solution to the impermeable wall problem of the spherically symmetric Navier--Stokes equations, which is the limit of the NSK equations as $\kappa\to0$. \\

\noindent {\bf Proof of Theorem \ref{thm:vanishing_capillarity} (1)}:
We recall that the perturbation $\phi^\kappa(r)=\tilde{\rho}^\kappa(r)-\rho_+$ satisfies the following equation (see \eqref{eq:phi_impermeable}):
	\begin{equation}\label{eq:phi_impermeable_2}
		\kappa(r^{n-1}\phi^\kappa_r)_r = r^{n-1}(h(\phi^\kappa+\rho_+)-h(\rho_+)).
	\end{equation}
	We multiply \eqref{eq:phi_impermeable_2} by $\phi^\kappa$, integrating over $r\in [1,\infty)$, and then take integration by parts to get
	\begin{equation}\label{eq:phi^k}
		-\kappa\rho_b\phi^\kappa(1)-\kappa\int_1^\infty r^{n-1}(\phi^\kappa_r)^2\,dr=\int_1^\infty r^{n-1}(h(\phi^\kappa+\rho_+)-h(\rho_+))\phi^\kappa\,dr.
	\end{equation}
	Since $h(\rho)$ defined in \eqref{eq:h} is an increasing function, there exists a positive constant $C_0$ such that
	\[(h(\phi^\kappa+\rho_+)-h(\rho_+))\phi^\kappa \ge C_0 (\phi^\kappa)^2.\]
	Thus, we rearrange the terms in \eqref{eq:phi^k} to get
	\begin{equation}\label{est:vanishing_kappa_1}
		C_0\int_1^\infty r^{n-1}(\phi^\kappa)^2\,dr+\kappa\int_1^\infty r^{n-1}(\phi^\kappa_r)^2\,dr=-\kappa\rho_b\phi^\kappa(1)\le \kappa|\rho_b||\phi^\kappa(1)|.
	\end{equation}
	To control the right-hand side of \eqref{est:vanishing_kappa_1}, we use the fact that $r\ge1$ and the H\"older inequality to get
	\begin{align*}
		|\phi^\kappa(1)|^2&=-\int_1^\infty \frac{d}{dr}(\phi^\kappa)^2\,dr=-2\int_1^\infty \phi^\kappa\phi^\kappa_r\,dr\\
		&\le 2\left(\int_1^\infty r^{n-1}(\phi^\kappa)^2\,dr\right)^{\frac{1}{2}}\left(\int_1^\infty r^{n-1}(\phi^\kappa_r)^2\,dr\right)^{\frac{1}{2}}.
	\end{align*}
	Thus, we use Young's inequality to further estimate \eqref{est:vanishing_kappa_1} as
	\begin{align}
		\begin{aligned}\label{est:vanishing_kappa_2}
			C_0&\int_1^\infty r^{n-1}(\phi^\kappa)^2\,dr+\kappa\int_1^\infty r^{n-1}(\phi^\kappa_r)^2\,dr\\
			&\le \sqrt{2}\kappa|\rho_b|\left(\int_1^\infty r^{n-1}(\phi^\kappa)^2\,dr\right)^{\frac{1}{4}}\left(\int_1^\infty r^{n-1}(\phi^\kappa_r)^2\,dr\right)^{\frac{1}{4}}\\
			&\le \frac{C_0}{4}\int_1^\infty r^{n-1}(\phi^\kappa)^2\,dr+\frac{\kappa}{4}\int_1^\infty r^{n-1}(\phi^\kappa_r)^2\,dr + C|\rho_b|^2\kappa^{3/2}.
		\end{aligned}
	\end{align}
	Again, rearranging the terms in \eqref{est:vanishing_kappa_2}, we finally get
	\[C_0\int_1^\infty r^{n-1}(\phi^\kappa)^2\,dr+\kappa\int_1^\infty r^{n-1}(\phi^\kappa_r)^2\,dr\le C\kappa^{3/2},\]
	which implies the desired estimates for $L^2$-norms. Furthermore, using Sobolev interpolation inequality, we get
	\[\|\phi^\kappa\|^2_{L^\infty}\le C\|\phi^\kappa\|_{L^2_r(1,\infty)}\|\phi^\kappa_r\|_{L^2_r(1,\infty)}\le C\kappa^{3/4}\kappa^{1/4}\le C\kappa.\]
This completes the proof of Theorem \ref{thm:vanishing_capillarity} (1). \qed
\subsection{Singularly scaled capillarity boundary condition}
When the boundary condition $\rho_b$ is inversely proportional to $\sqrt{\kappa}$, say $\rho_b=\rho_{b}^0/\sqrt{\kappa}$ with fixed $\rho^0_b$, we can expect that the limit solution $\bar{\rho}$ becomes non-trivial, and it satisfies the following second-order boundary value problem:
\begin{equation}\label{eq:limit_singular}
	\bar{\rho}_{yy} = h(\bar{\rho})-h(\rho_+),\quad \bar{\rho}_y(0)=\rho_b^0,\quad\lim_{y\to\infty}\bar{\rho}(y)=\rho_+.
\end{equation}

Therefore, it is expect that the solution $\tilde{\rho}^\kappa(r)$ will converge to the limit $\bar{\rho}(y)$, which is the statement of Theorem \ref{thm:vanishing_capillarity} (2). Before we prove Theorem \ref{thm:vanishing_capillarity} (2), we first investigate the existence and properties of the solution to \eqref{eq:limit_singular}.

\begin{lemma}\label{lem:limit}
	Suppose $|\rho^0_b|$ is sufficiently small. Then, there exists a unique solution $\bar{\rho}(y)$ to \eqref{eq:limit_singular} satisfying
	\[|\bar{\rho}(y)-\rho_+|\le C|\rho^0_b|e^{-\sqrt{h'(\rho_+)}y},\quad |\bar{\rho}_y(y)|\le C|\rho^0_b|e^{-\sqrt{h'(\rho_+)}y},\]
	for some positive constant $C$.
\end{lemma}

\begin{proof}
	We reduce the second-order differential equation \eqref{eq:limit_singular} to the system of first-order differential equations by introducing $x_1 :=\bar{\rho}$ and $x_2:=\bar{\rho}_y$:
	\begin{align*}
		&\frac{d x_1}{dy} = x_2,\quad \frac{dx_2}{dy}=h(x_1)-h(\rho_+),\\
		&\lim_{y\to\infty}x_1(y) = \rho_+,\quad x_2(0)=\rho_b^0.
	\end{align*}
	It is easy to observe that the solution to the above system of ODEs preserves the energy $\frac{1}{2}x_2^2-W(x_1)$, where
	\[W(x):=\int_{\rho_+}^x (h(y)-h(\rho_+))\,dy.\]
	It is clear that $W'(x) = h(x)-h(\rho_+)$ and $W''(x) = h'(x) = \frac{P'(x)}{x}>0$. In particular, the eigenvalues of the Jacobian at the equilibrium $(\rho_+,0)$ are $\pm\sqrt{h'(\rho_+)}$ and therefore, the equilibrium $(\rho_+,0)$ is a saddle point. Hence, in order to find the solution to \eqref{eq:limit_singular}, we need to verify that there exists an initial value $x_1(0)$ such that the point $(x_1(0),x_2(0))$ is on the stable manifold, $x_2 = -\textup{sgn}(x_1-\rho_+) \sqrt{2W(x_1)}=:S(x_1)$. Here, the sign is uniquely determined since $x_2(0)=\rho^0_b$ and $x_1-\rho_+$ has opposite sign. However, as $W(\rho_+)=W'(\rho_+)=0$, we have $W(x_1) = \frac{W''(\rho_+)}{2}|x_1-\rho_+|^2+O(|x_1-\rho_+|^3)$ where $x_1$ is close to $\rho_+$. Thus,
	\[S'(\rho_+) = \lim_{x_1\to\rho_+}\frac{S(x_1)-S(\rho_+)}{x_1-\rho_+}=\lim_{x_1\to\rho_+}\frac{-\textup{sgn}(x_1-\rho_+)\sqrt{2W(x_1)}}{x_1-\rho_+}=-\sqrt{W''(\rho_+)}\neq0.\]
	Thus, by inverse function theorem, for sufficiently small $|\rho_b^0|$, there exists a unique $\rho_-$ such that the point $(\rho_-,\rho_b^0)$ is on the stable manifold, and therefore, there exists a unique solution to \eqref{eq:limit_singular} such that $\bar{\rho}(0)=\rho_-$, $\bar{\rho}_y(0)=\rho_b^0$ and $\lim_{y\to\infty} \bar{\rho}(y)=\rho_+$. 
	Now, we derive the decay rate of $\bar{\rho}$. For simplicity, we only consider the case when $\rho_b^0<0$, and therefore, $\bar{\rho}-\rho_+>0$. Since the trajectory is on the stable manifold, $\bar{\rho}$ should satisfies the following reduced ODE:
	\[\bar{\rho}_y = -\sqrt{2W(\bar{\rho})}.\]
	However, we already observed that $W(\bar{\rho})=\frac{W''(\rho_+)}{2}|\bar{\rho}-\rho_+|^2+O(|\bar{\rho}-\rho_+|^3)$ as $y\to\infty$, we have
	\[\bar{\rho}_y =-\sqrt{W''(\rho_+)}(\bar{\rho}-\rho_+)+o(|\bar{\rho}-\rho_+|).\]
	This gives an exponential decay rate for $\bar{\rho}$:
	\[|\bar{\rho}(y)-\rho_+| \le C|\bar{\rho}(0)-\rho_+|e^{-\sqrt{W''(\rho_+)}y}\le C|\rho_b^0|e^{-\sqrt{h'(\rho_+)}y}.\]
\end{proof}

\noindent {\bf Proof of Theorem \ref{thm:vanishing_capillarity} (2)}:
Again, we recall that the stationary solution $\tilde{\rho}^\kappa$ satisfies
	\[\kappa(r^{n-1}\tilde{\rho}^\kappa_r)_r=r^{n-1}(h(\tilde{\rho}^\kappa)-h(\rho_+)).\]
Furthermore, $\bar{\rho}^\kappa(r):=\bar{\rho}(\frac{r-1}{\sqrt{\kappa}})$ satisfies
	\begin{align*}
		\kappa(r^{n-1}\bar{\rho}^\kappa_r)_r&=\kappa r^{n-1}\bar{\rho}^\kappa_{rr}+\kappa (n-1)r^{n-2}\bar{\rho}^\kappa_r=r^{n-1}\bar{\rho}_{yy}+\sqrt{\kappa}\frac{n-1}{r}r^{n-1}\bar{\rho}_y\\
		&=r^{n-1}(h(\bar{\rho})-h(\rho_+))+\sqrt{\kappa}\frac{n-1}{r}r^{n-1}\bar{\rho}_y\left(\frac{r-1}{\sqrt{\kappa}}\right).
	\end{align*}
	Therefore, taking difference between two equations, we get the following equation for the perturbation $\phi^\kappa:=\tilde{\rho}^\kappa-\bar{\rho}^\kappa$:
	\[\kappa(r^{n-1}\phi^\kappa_r)_r = r^{n-1}(h(\phi^\kappa+\bar{\rho}^\kappa)-h(\bar{\rho}^\kappa))+\sqrt{\kappa}\frac{n-1}{r}r^{n-1}\bar{\rho}_y\left(\frac{r-1}{\sqrt{\kappa}}\right).\]
	As before, we multiply $\phi^\kappa$ on the both sides, and take integration by parts and the boundary conditions
	\[\tilde{\rho}^\kappa_r(1) = \frac{\rho_b^0}{\sqrt{\kappa}},\quad \bar{\rho}^\kappa_r(1) = \frac{1}{\sqrt{\kappa}}\bar{\rho}_y(0)=\frac{\rho_b^0}{\sqrt{\kappa}},\quad\mbox{hence}\quad \phi^\kappa_r(1)=0\]
	to obtain
	\[-\kappa\int_1^\infty (\phi^\kappa_r)^2 r^{n-1}\,dr=\int_{1}^\infty (h(\phi^\kappa+\bar{\rho}^\kappa)-h(\bar{\rho}^\kappa))\phi^\kappa r^{n-1}\,dr+\sqrt{\kappa}\int_1^\infty \phi^\kappa \frac{n-1}{r}r^{n-1}\bar{\rho}_y\,dr.\]
	Again, we use the monotonicity of $h$ to guarantee that there exists a positive constant $C_0$ that is independent of $\kappa$ such that $(h(\phi^\kappa+\bar{\rho}^\kappa)-h(\bar{\rho}^\kappa))\phi^\kappa\ge C_0(\phi^\kappa)^2$, which implies
	\begin{align*}
		\kappa\int_1^\infty& (\phi^\kappa_r)^2 r^{n-1}\,dr + C_0\int_1^\infty (\phi^\kappa)^2 r^{n-1}\,dr\\
		&\le \sqrt{\kappa}\int_1^\infty |\phi^\kappa|\frac{n-1}{r}r^{n-1}|\bar{\rho}_y|\,dr\le \frac{C_0}{2}\int_1^\infty (\phi^\kappa)^2r^{n-1}\,dr+C\kappa\int_1^{\infty}\frac{(n-1)^2}{r^2}|\bar{\rho}_y|^2r^{n-1}\,dr.
	\end{align*}
	This yields
	\[\kappa\int_1^\infty (\phi^\kappa_r)^2r^{n-1}\,dr +\frac{C_0}{2}\int_1^\infty (\phi^\kappa)^2r^{n-1}\,dr\le C\kappa^{3/2}\int_0^\infty |\bar{\rho}_y(y)|^2(1+\sqrt{\kappa}y)^{n-3}\,dy\le C\kappa^{3/2},\]
	where we used the decay estimate for $\bar{\rho}_y$ in Lemma \ref{lem:limit} in the last inequality. Hence, we conclude the desired convergence rate:
	\[\|\phi^\kappa\|_{L^2_r(1,\infty)}\le C\kappa^{3/4},\quad \|\phi^\kappa_r\|_{L^2_r(1,\infty)}\le C\kappa^{1/4},\quad \|\phi^\kappa\|_{L^\infty(1,\infty)}\le C\kappa^{1/2}.\]
This completes the proof of Theorem \ref{thm:vanishing_capillarity} (2).
\qed

\section{Numerical experiments}\label{sec:numeric}
\setcounter{equation}{0}
In this section, we present the numerical solutions to the impermeable wall problems \eqref{NSK_stationary_impermeable}--\eqref{boundary_impermeable} for both fixed and singularly scaled boundary condition $\rho_b$, validating the convergence rates in Theorem \ref{thm:vanishing_capillarity}. 

\subsection{Fixed boundary condition}
For the numerical simulation of the \eqref{NSK_stationary_impermeable}--\eqref{boundary_impermeable} with the fixed boundary data $\rho_b$, we choose the parameters as
\[n=3,\quad \rho_+=1,\quad \rho_b = -1,\quad \gamma=1.\]
As in Theorem \ref{thm:vanishing_capillarity}, the solution $\tilde{\rho}$ converges to the constant state $\rho_+$ as $\kappa\to0$. Therefore, we measure both the $L^2$-norm $\|\tilde{\rho}^\kappa-\rho_+\|_{L^2_r}$ and $L^\infty$-norm $\|\tilde{\rho}^\kappa-\rho_+\|_{L^\infty}$ of the errors and estimates their convergence rates. Figure \ref{fig:convergence-1} shows the result of the numerical simulation. The left figure shows the convergence rates for both $L^2$ and $L^\infty$ errors coincide with the theoretical results $O(\kappa^{3/4})$ and $O(\kappa^{1/2})$ as $\kappa\to0$. This shows that the convergence rate obtained in Theorem \ref{thm:vanishing_capillarity} is indeed optimal convergence rate. The right figure shows that the profiles of the stationary solutions $\tilde{\rho}^\kappa$ converges to the constant $\rho_+$ as $\kappa\to0$. 

\begin{figure}[h!]
	\includegraphics[width=1\textwidth]{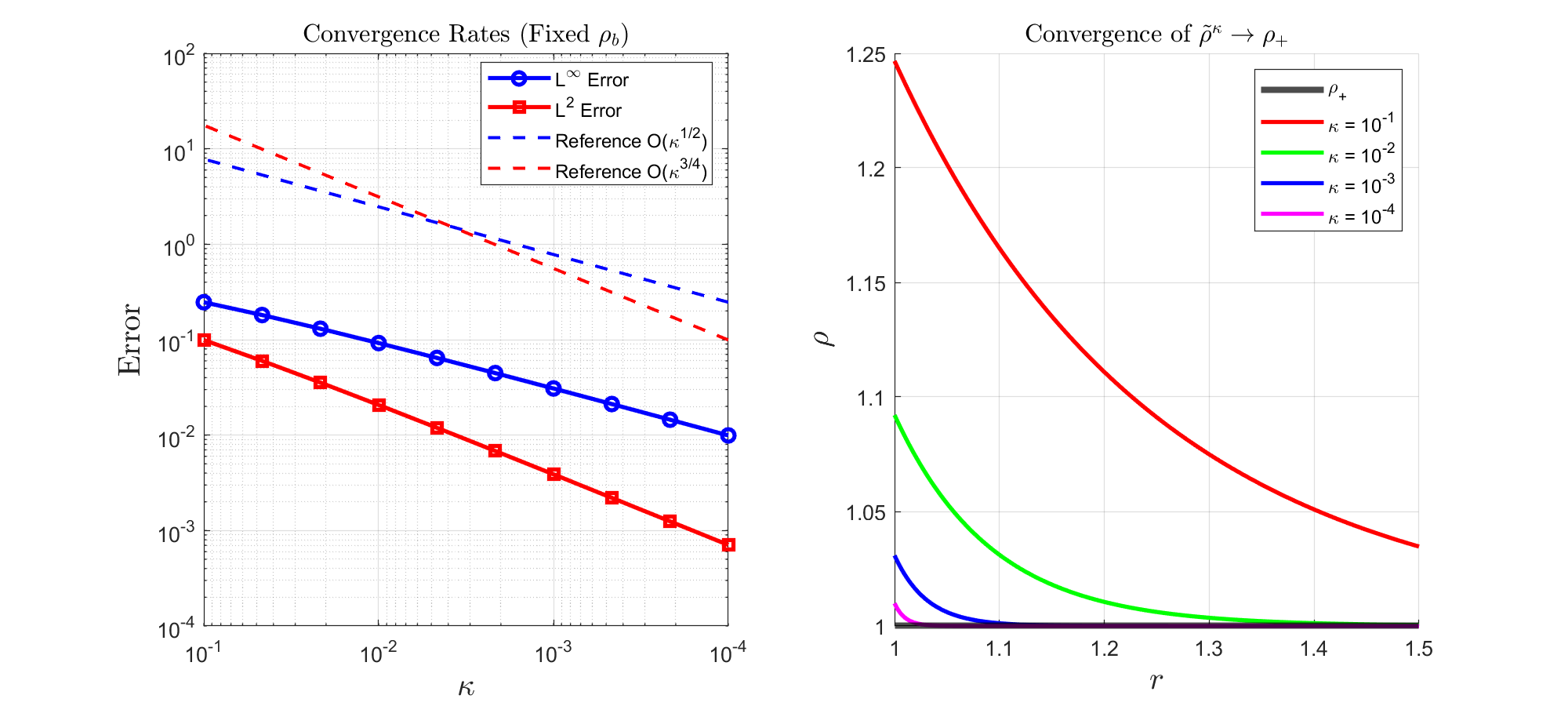}
	\caption{(Left) Convergence rates of $\|\tilde{\rho}^\kappa-\rho_+\|_{L^2}$ and $\|\tilde{\rho}^\kappa-\rho_+\|_{L^\infty}$ for the fixed boundary condition. The numerical convergence rates coincide with the theoretical ones. (Right) Convergence of the profile $\tilde{\rho}^\kappa$ towards the constant $\rho_+$ for the fixed boundary condition.}
	\label{fig:convergence-1}
\end{figure}

\subsection{Singularly scaled boundary condition}
Now, we consider the case when the boundary condition is singularly scaled as $\rho_b = \rho_b^0/\sqrt{\kappa}$. In this case, we choose $\rho_b^0=-0.1$, and the other parameters remain the same as the fixed boundary condition case. Figure \ref{fig:convergence_2-1} shows the desired convergence rates of $\|\tilde{\rho}^\kappa-\bar{\rho}^\kappa\|_{L^2_r}$ and $\|\tilde{\rho}-\bar{\rho}^\kappa\|_{L^\infty}$, which are $O(\kappa^{3/4})$ and $O(\kappa^{1/2})$ as $\kappa\to0$ respectively. Again, this shows that our theoretical convergence rate is optimal. Finally, we depict the profile of the stationary solutions $\tilde{\rho}^\kappa$ for various values of $\kappa$, both for the physical domain $r\in(1,\infty)$ and the rescaled domain $y\in(0,\infty)$. As the result of Theorem \ref{thm:vanishing_capillarity} and Remark \ref{rem:y} imply, the rescaled stationary solution $\tilde{\rho}^\kappa(1+\sqrt{\kappa}y)$ converges to $\bar{\rho}(y)$ as $\kappa\to0$. This convergence can be observed in the right figure of Figure \ref{fig:convergence_2-2}.

\begin{figure}[h!]
	\includegraphics[width=0.8\textwidth]{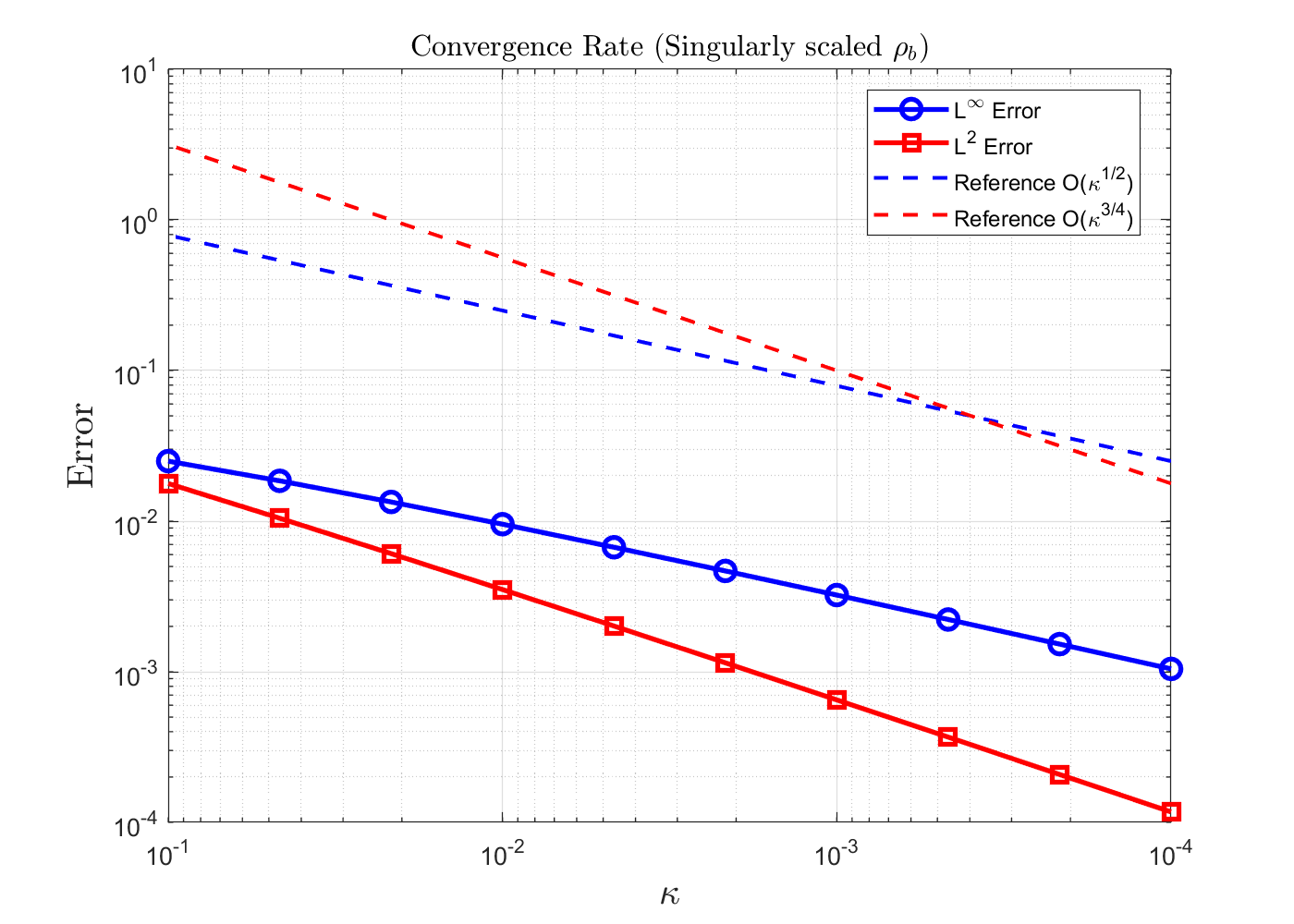}
	\caption{Convergence rates of $\|\tilde{\rho}^\kappa-\bar{\rho}^\kappa\|_{L^2}$ and $\|\tilde{\rho}^\kappa-\bar{\rho}^\kappa\|_{L^\infty}$ for the singularly scaled boundary condition. The numerical convergence rates coincide with the theoretical ones.}
	\label{fig:convergence_2-1}
\end{figure}
\begin{figure}[h!]
	\includegraphics[width=1\textwidth]{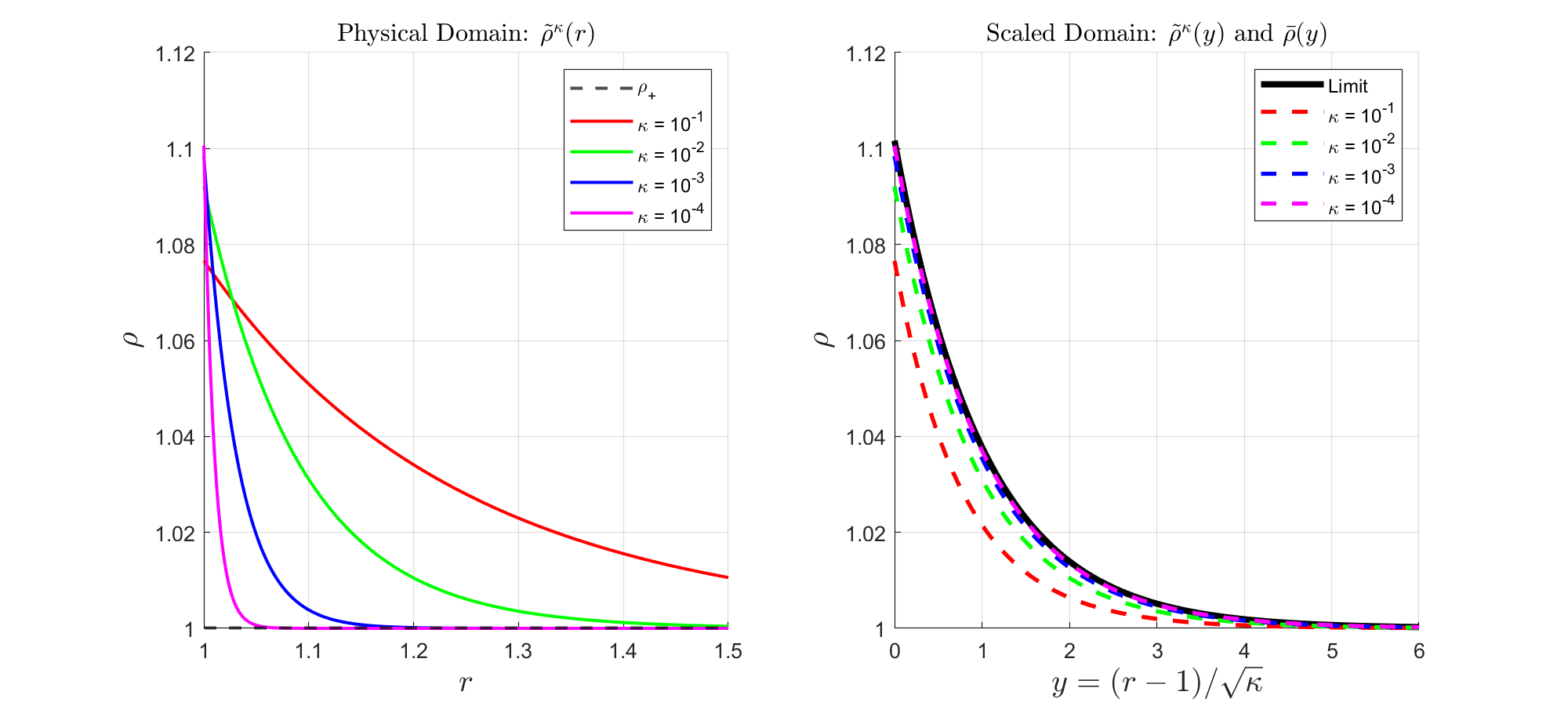}
	\caption{(Left) Profiles of the stationary solutions $\tilde{\rho}^\kappa$ with respect to the physical domain $r\in(1,\infty)$. (Right) Profiles of the stationary solutions $\tilde{\rho}^\kappa$ with respect to the scaled domain $y\in (0,\infty)$. The solution profile $\tilde{\rho}^\kappa(1+\sqrt{\kappa}y)$ converges to the limit $\bar{\rho}(y)$ as $\kappa\to0$.}
	\label{fig:convergence_2-2}
\end{figure}

\appendix

\section{Spherical Helmholtz equation and the properties of the modified Bessel functions}\label{sec:app}
\setcounter{equation}{0}

In this appendix, we present the details on how the general solution to the spherically symmetric Helmholtz equation
\begin{equation}\label{spherical_Helmholtz}
	\phi_{rr}+\frac{n-1}{r}\phi_r-\alpha^2\phi=0
\end{equation}
can be expressed as a linear combinations of the functions in \eqref{basis_solutions}. First of all, we introduce $\psi(r) := r^{\nu}\phi(r)$, where $\nu = \frac{n-2}{2}$. After substituting $\phi(r) = r^{-\nu}\psi(r)$, we get the equation for $\psi$ as
\[r^2\psi_{rr}+r\psi_r -\left(\alpha^2r^2+\nu^2\right)\psi=0.\]
Now, using the change of variable $s=\alpha r$, $\psi$ satisfies
\[s^2\psi_{ss}+s\psi_s-(s^2+\nu^2)\psi=0,\]
which is nothing but the modified Bessel equations, whose general solution is given by the linear combinations of the modified Bessel functions of first kind $I_\nu(s)$ and second kind $K_\nu(s)$. Thus, considering the change of variable, the general solution to \eqref{spherical_Helmholtz} is represented as a linear combination of
$(\alpha r)^{-\nu}I_\nu(\alpha r)$ and $(\alpha r)^{-\nu}K_{\nu}(\alpha r)$.

It is known that the derivatives of the modified Bessel functions satisfy

\begin{equation}\label{property_Bessel_1}
	\frac{d}{dz}(z^{-\nu}I_\nu(z))=z^{-\nu}I_{\nu+1}(z),\quad \frac{d}{dz}(z^{-\nu} K_{\nu}(z)) = -z^{-\nu}K_{\nu+1}(z).
\end{equation}
Furthermore, the modified Bessel functions has an asymptotic properties
\[I_{\nu}(z)\sim \frac{e^z}{\sqrt{2\pi z}},\quad K_\nu(z)\sim \sqrt{\frac{\pi}{2 z}}e^{-z}\quad\mbox{as}\quad z\to\infty\]
for any $\nu>-1$. This implies, for any fixed $\nu$ and $z>z_0$, there exists a constant $C$, only depending on $\nu$, such that
\begin{equation}\label{property_Bessel_2}
	\frac{C^{-1}e^{x}}{\sqrt{x}}\le I_\nu(x)\le \frac{Ce^{x}}{\sqrt{x}},\quad \frac{C^{-1}e^{-x}}{\sqrt{x}}\le K_{\nu}(x)\le \frac{Ce^{-x}}{\sqrt{x}}.
\end{equation}
This gives the desired properties of $I_\nu$ and $K_\nu$. We now explain how to derive the formula \eqref{expression_phi} for solving the equation
\[\phi'' +\frac{n-1}{r}\phi'-\alpha^2 \phi = F(r),\quad '=\frac{d}{dr},\]
subject to the boundary conditions
\[\phi'(1) = 0,\quad \lim_{r\to\infty}\phi(r)=0.\]
We find the solution $\phi$ by using the method of variation of parameters as follow. Assume the following ansatz of $\phi$
\[\phi(r) = u(r)\phi_+(r)+v(r)\phi_-(r),\]
under the assumption
\begin{equation}\label{uv1}
	u'(r)\phi_+(r)+v'(r)\phi_-(r)=0.
\end{equation}
After taking derivative on $\phi$, we have
\[\phi'(r)=u(r)\phi'_+(r)+v(r)\phi'_-(r),\quad \phi''(r) = u'(r)\phi'_+(r)+v'(r)\phi'_-(r)+u(r)\phi''_+(r)+v(r)\phi''_-(r),\]
and therefore,
\begin{equation}\label{uv2}
	\phi''+\frac{n-1}{r}\phi'-\alpha^2\phi=u'(r)\phi_+'(r)+v'(r)\phi'_-(r)=F(r).
\end{equation}
Thus, combining \eqref{uv1} and \eqref{uv2}, we conclude that $u$ and $v$ satisfy
\[\begin{pmatrix}
	\phi_+ & \phi_-\\ \phi_+' & \phi'_-
\end{pmatrix}\begin{pmatrix}
u'\\v'
\end{pmatrix}=\begin{pmatrix}
0\\F
\end{pmatrix},\]
or in other words,
\[u'(r) = -\frac{\phi_-(r)F(r)}{W(r)},\quad v'(r) = \frac{\phi_+(r)F(r)}{W(r)}.\]
After integration, $\phi$ can be represented as

\[\phi(r)=u(r)\phi_+(r)+v(r)\phi_-(r)=\phi_+(r)\int_{a}^r -\frac{\phi_-(s)F(s)}{W(s)}\,ds+\phi_-(r)\int_{b}^r \frac{\phi_+(s)F(s)}{W(s)}\,ds, \]
for some constants $a$ and $b$. Now, we determine the constants $a$ and $b$ to match the boundary condition. First, we need $\phi'(1)=0$, which implies
\[u(1)\phi_+'(1)+v(1)\phi'_-(1)=0,\]
and from the condition $\phi_+'(1)=0$, we get $v(1)=0$. Hence, we have $b=1$. Next, in order to satisfies the condition $\lim_{r\to\infty}\phi(r)=0$, we need the coefficient of $\phi_+(r)$ should be 0 as $r\to+\infty$. This yields $a=+\infty$. In conclusion, the solution should be
\begin{align*}
	\phi(r)&=\phi_+(r)\int_r^\infty \frac{\phi_-(s)F(s)}{W(s)}\,ds+\phi_-(r)\int_1^r \frac{\phi_+(s)F(s)}{W(s)}\,ds\\
	&=\int_1^\infty G(r,s)F(s)s^{n-1}\,ds,
\end{align*}
which is the desired formula.

\end{document}